\documentclass[journal,twoside,web]{ieeecolor}

\usepackage{generic}
\usepackage{cite}

\usepackage{amsthm,amsmath,amssymb,amsfonts}
\usepackage{bm}
\usepackage{algorithmic}
\usepackage[ruled, noline]{algorithm2e}
\usepackage{graphicx}

\usepackage{array}
\usepackage{multirow}
\usepackage{diagbox}

\usepackage{subcaption}

\usepackage{textcomp}
\def\BibTeX{{\rm B\kern-.05em{\sc i\kern-.025em b}\kern-.08em
T\kern-.1667em\lower.7ex\hbox{E}\kern-.125emX}}

\usepackage{lipsum}


\theoremstyle{plain}

\newtheorem{lemma}{Lemma}
\newtheorem{corollary}{Corollary}

\theoremstyle{definition}

\def\({\left(}
\def\){\right)}
\def\[{\left[}
\def\]{\right]}

\def\abf{{\bf a}}
\def\bbf{{\bf b}}
\def\dbf{{\bf d}}
\def\Fbf{{\bf F}}  
\def\Gbf{{\bf G}}  

\def\Jbf{{\bf J}}  
\def\Kbf{{\bf K}}  
\def\Mbf{{\bf M}}  
\def\Nbf{{\bf N}}  
\def\Pbf{{\bf P}}  
\def\Qbf{{\bf Q}}  
\def\Rbf{{\bf R}}  
\def\Sbf{{\bf S}}  
\def\Tbf{{\bf T}}  
\def\Ubf{{\bf U}}  \def\ubf{{\bf u}}
\def\Vbf{{\bf V}}  
\def\Wbf{{\bf W}}  \def\wbf{{\bf w}}
\def\xbf{{\bf x}}
\def\Ybf{{\bf Y}}  \def\ybf{{\bf y}}
\def\Zbf{{\bf Z}}  \def\zbf{{\bf z}}
\def\Deltabf{\bm{\Delta}}  \def\deltabf{\bm{\delta}}
\def\etabf{{\bm{\eta}}}
\def\xibf{{\bm{\xi}}}
\def\Phibf{\bm{\Phi}}
\def\Psibf{\bm{\Psi}}
\def\Ccal{\mathcal{C}}  
\def\Dcal{\mathcal{D}}  
\def\Hcal{\mathcal{H}}  
\def\Kcal{\mathcal{K}}  
\def\Rcal{\mathcal{R}}  

\newif\ifshowWriterComment
\newcommand\writercomment[3]{\expandafter\newcommand\csname #2\endcsname[1]{\ifshowWriterComment{\color{#3} (#1: ##1)}\fi}}

\newcommand{\SFpair}{\{\Phibf_\xbf, \Phibf_\ubf\}}
\newcommand{\OFqple}{\{\Phibf_{\xbf\xbf}, \Phibf_{\ubf\xbf}, \Phibf_{\xbf\ybf}, \Phibf_{\ubf\ybf}\}}

\newcommand{\hatOFqple}{\{\hat{\Phibf}_{\xbf\xbf}, \hat{\Phibf}_{\ubf\xbf}, \hat{\Phibf}_{\xbf\ybf}, \hat{\Phibf}_{\ubf\ybf}\}}

\newcommand{\norm}[1]{\left\lVert #1 \right\rVert}

\def\mat#1{\begin{bmatrix}#1\end{bmatrix}}
\def\t{[t]}

\def\tm{[t-1]}

\def\cor#1{Corollary~\ref{cor:#1}}
\def\fig#1{Fig.~\ref{fig:#1}}
\def\subfig#1#2{Fig.~\ref{fig:#1}(\subref{subfig:#1-#2})}
\def\lem#1{Lemma~\ref{lem:#1}}
\def\sec#1{Section~\ref{sec:#1}}

\def\eqn#1{\eqref{eqn:#1}}

\def\st{{\rm s.t.}}
\def\OptConsSep{&&\quad}

\newcommand\OptCons[3]{
&\ #1
\ifx\\#2\\ \else \OptConsSep #2 \fi%
\ifx\\#3\\ \nonumber \else \label{eqn:#3} \fi%
}

\newcommand{\OptMinN}[2]{
\begin{alignat*}{2}
\min\ &\ #1 \\
\st\ #2
\end{alignat*}
}

\def\xx{\xbf\xbf}  \def\xu{\xbf\ubf}  \def\xy{\xbf\ybf}
\def\ux{\ubf\xbf}  \def\uu{\ubf\ubf}  \def\uy{\ubf\ybf}  \def\uz{\ubf\zbf} \def\uw{\ubf\wbf}
\def\yx{\ybf\xbf}  \def\yu{\ybf\ubf}  \def\yy{\ybf\ybf}  \def\yz{\ybf\zbf}  \def\yw{\ybf\wbf}

\def\zy{\zbf\ybf}  \def\zu{\zbf\ubf}  \def\zw{\zbf\wbf}  \def\zz{\zbf\zbf}
\def\wy{\wbf\ybf}  \def\wu{\wbf\ubf}  \def\ww{\wbf\wbf}  \def\wz{\wbf\zbf}

\def\xd{\xbf\deltabf}
\def\ud{\ubf\deltabf}
\def\dx{\deltabf\xbf}  \def\du{\deltabf\ubf}  \def\dd{\deltabf\deltabf}

\def\Rtru{\Rbf(\Deltabf)}
\def\Stru{\Sbf(\Deltabf)}
\def\Rnm{\hat{\Rbf}}
\def\Snm{\hat{\Sbf}}

\title{Realization-Stability Lemma \\for Controller Synthesis}

\author{Shih-Hao~Tseng
\thanks{
This paper was presented in part at IEEE American Control Conference, May 26--28, 2021 and at IEEE Conference on Decision and Control, December 13--15, 2021.
}
\thanks{Shih-Hao Tseng was with Division of Engineering and Applied Science, California Institute of Technology, Pasadena, CA 91125 USA.}
}

\showWriterCommenttrue

\begin{document}

\maketitle

\bstctlcite{IEEE_BSTcontrol}

\begin{abstract}

We have witnessed the emergence of several controller parameterizations and the corresponding synthesis methods, including Youla, system level, input-output, and many other new proposals. Meanwhile, under the same synthesis method, there are multiple realizations to adopt. Different synthesis methods/realizations target different plants/scenarios. Also, various robust results are proposed to deal with different perturbed system structures. Except for some case-by-case studies, we don't currently have a unified framework to understand their relationships.

To address the issue, we show that existing controller synthesis methods and realization proposals are all special cases of a simple lemma, the realization-stability lemma. The lemma leads to easier equivalence proofs among existing methods and robust stability conditions for general system analysis. It also enables the formulation of a general controller synthesis problem, which provides a unified foundation for controller synthesis, realization derivation, and robust stability analysis.

\end{abstract}


\section{Introduction}\label{sec:introduction}

\IEEEPARstart{C}{ontroller} synthesis is one of the core missions in control theory. It aims to derive controllers that could stabilize a given plant in the presence of \emph{disturbance} and/or \emph{perturbation}. Synthesizing such controllers, the internally stabilizing controllers, is highly non-trivial, especially as the modern systems grow larger, more complex, and involve more input/output signals.
Meanwhile, the ever-increasing system complexity also demands sophisticated and efficient controller synthesis techniques and drives the research of controller synthesis theories.

The first challenge for controller synthesis is \emph{external disturbance rejection}, i.e., the synthesized controller should be able to neutralize the impact of unanticipated external disturbances in the long run.
A well-celebrated pioneer work on this direction is by Youla et al. \cite{youla1976modern1, youla1976modern2}, which shows that the set of all internally stabilizing controllers can be parameterized using a coprime factorization approach. One drawback of Youla parameterization is the difficulty of imposing structural constraints -- the constraints could only be imposed (while maintaining convexity)
in an intricate form admitting quadratic invariant property \cite{rotkowitz2005characterization,sabau2014youla,lessard2015convexity}.
To address this issue, system level parameterization (SLP) \cite{anderson2019system,wang2019system} proposes to work on the closed-loop system response and the corresponding system level synthesis (SLS) method can easily incorporate multiple structural constraints into a much simpler convex program \cite{wang2016localized,anderson2017structured}.
The success of SLP triggers the study of affine space parameterization of internally stabilizing controllers. \cite{furieri2019input} shows that the set of internally stabilizing controllers can also be parameterized in an input-output manner using the input-output parameterization (IOP).
Though a recent paper shows that Youla, SLP, and IOP are equivalent \cite{zheng2020equivalence}, there are still new affine space parameterizations found \cite{zheng2019system}.

Besides rejecting external disturbances, more advanced controller synthesis theories study perturbed plants and aim to achieve \emph{robust stability}, where the controller could still stabilize the underlying plant even when perturbation, or sometimes referred as \emph{uncertainty}, occurs. There are many robust results proposed in the literature, and they mainly focus on mitigating one of the following two concerns -- plant uncertainty or controller perturbation.
Due to the estimation precision, dimension limit, or pliant nature, the plant model could differ from its true dynamics. To deal with this uncertainty, one category of robust results aims to synthesize a controller that can stabilize a set of plants, such as $\mu$-synthesis \cite{doyle1982analysis, doyle1985structured, zhou1998essentials}, robust primal-dual Youla parameterization \cite{niemann2002reliable}, and robust input-output parameterization (IOP) \cite{zheng2020sample}.
On the other hand, even with an exact plant model, realization of the synthesized controller may still deviate from its desired form because of resolution restriction. As such, another category of robust results ensures a perturbed controller realization can still stabilize the plant, e.g., robust system level synthesis (SLS) \cite{matni2017Scalable,anderson2019system,boczar2018finite}.

Given the flourishing development of novel parameterizations and their corresponding synthesis methods, we have some natural questions to ask: Have we exhausted all possible parameterizations? Will we discover new synthesis methods? If so, why would they be the way they are? And, perhaps more importantly, how could we find/understand them systematically?

To add to this already puzzling situation, we have seen new results on realizations. Realizations, or block diagrams/implementations,\footnote{We adopt the terminology in \cite{tseng2020deployment,tsengsubsynthesis} that distinguishes ``realizations'' from ``implementations,'' where the former refers to the block diagrams (mathematical expressions) and the latter is reserved for the physical architecture consisting of computation, memory, and communication units.}
describe how a system can be built from some interconnection of basic blocks/transfer functions.
It is well known, also shown by recent studies \cite{tseng2020deployment,li2020separating}, that the same controller can admit multiple different realizations, even under the same parameterization scheme.
We would then wonder if we can only handle those realizations individually, or if there is a unified framework to study them.

Similarly, robust controller synthesis results are derived via various analysis tools targeting distinct settings. It is not straightforward to see how they relate with one another and how one method may be applicable for a different setting. As a consequence, most robust results are taught, learned, and applied in a case-by-case manner. Moreover, in addition to the two concerns above, one can easily imagine some compound scenarios where both the plant and controller are subject to perturbation. We would then wonder how to deal with diverse perturbation scenarios systematically. In particular, we are interested in a unified approach to robust controller analysis and synthesis.

\subsection{Contributions and Organization}

The main contribution of this paper is a systematic approach to all of the above seemingly unrelated questions through a simple \emph{realization-stability lemma} that relates closed-loop realizations with internal stability. The lemma enables us to formulate the general controller synthesis problem that can derive \emph{all} possible parameterizations, thus providing a systematic way to study controller synthesis problems. We show that existing methods on controller synthesis and realization are all special cases of the general formulation. In addition, the lemma reveals that the transformation of external disturbances can be seen as the derivation of an equivalent system. The concept of equivalent systems then enables easy proof of equivalence among synthesis methods.

Also, this paper provides such a unified approach through the robust stability conditions for general systems derived from the realization-stability lemma.
As both plant and controller are included in a realization, the uncertainty to any of them is deemed a perturbation of the realization, thereby unifying diverse robustness concerns into one coherent form.
We also specialize our result for additively perturbed realizations. The robust stability condition leads to the formulation of the general robust controller synthesis problem, and we demonstrate how to derive existing results in the literature using the condition. In addition, we show how new robust results for output-feedback SLS and IOP can be obtained easily from the condition.

The paper is organized as follows. In~\sec{R-S_lemma}, we derive the realization-stability lemma, introduce equivalent systems under transformations, and formulate the general controller synthesis problem. We then extend the lemma to robust settings and derive the robust stability conditions in~\sec{robust_analysis}, along with the corresponding formulation of the general robust controller synthesis problem.
Leveraging the realization-stability lemma and the robust stability conditions, we unify existing controller synthesis results in~\sec{controller_synthesis}, realization results in \sec{realizations}, and robust results in~\sec{existing_robust_results} as corollaries.
In addition, we demonstrate how to apply our lemmas to derive new results and discuss the application scope in~\sec{applications}. Finally, we conclude the paper in~\sec{conclusion}.

\subsection{Notation}

\def\Rp{\Rcal_{p}}
\def\Rsp{\Rcal_{sp}}
\def\RHinf{\Rcal\Hcal_{\infty}}

Let $\Rp$, $\Rsp$, and $\RHinf$ denote the set of proper, strictly proper, and stable proper transfer matrices, respectively, all defined according to the underlying setting, continuous or discrete.
Lower- and upper-case letters (such as $x$ and $A$) denote vectors and matrices respectively, while bold lower- and upper-case characters and symbols (such as $\ubf$ and $\Rbf$) are reserved for signals and transfer matrices. We denote by $I$ and $O$ the identity and all-zero matrices (with dimensions defined according to the context).

\def\xx{\xbf\xbf}  \def\xu{\xbf\ubf}  \def\xy{\xbf\ybf}
\def\ux{\ubf\xbf}  \def\uu{\ubf\ubf}  \def\uy{\ubf\ybf}
\def\yx{\ybf\xbf}  \def\yu{\ybf\ubf}  \def\yy{\ybf\ybf}

\def\xd{\xbf\deltabf}
\def\ud{\ubf\deltabf}
\def\dx{\deltabf\xbf}  \def\du{\deltabf\ubf}  \def\dd{\deltabf\deltabf}

\section{Realization-Stability Lemma}\label{sec:R-S_lemma}
To begin with, we define the realization and internal stability matrices to derive the realization-stability lemma. We then discuss the transformation of external disturbances and introduce the concept of equivalent systems. Using the realization-stability lemma, we propose the formulation of a general controller synthesis problem.

We remark that the results in this section are general: They apply to both discrete-time and continuous-time systems.

\subsection{Realization and Internal Stability}\label{sec:R-S_lemma-RS}
\begin{figure}
\centering
\includegraphics[scale=1]{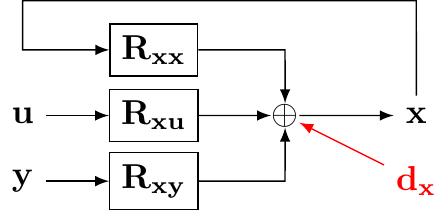}
\caption{The realization matrix $\Rbf$ describes each signal as a linear combination of the signals in the closed-loop system and the external disturbance $\dbf$. In the following figures of realizations, we omit drawing the additive disturbance $\dbf$ for simplicity.}
\label{fig:R-S-realization}
\end{figure}

We consider a closed-loop linear system with internal state $\etabf$ and external disturbance $\dbf$. The system operates according to the \emph{realization matrix} $\Rbf$:
\begin{align}
\etabf = \Rbf \etabf + \dbf.
\label{eqn:R}
\end{align}
$\etabf$ summarizes \emph{all} signals in the system. For instance, a state-feedback system might have $\etabf = \mat{\xbf\\ \xibf\\ \ubf}$ where $\xbf$ is the state, $\xibf$ is the internal state, and $\ubf$ is the control. For a given signal $\abf$, we denote by $e_{\abf}$ the column block that is identity at the rows corresponding to $\abf$ in $\etabf$. As a result,
$\etabf = \sum\limits_{\abf} e_{\abf} \abf$.

$\Rbf$ describes each signal as a linear combination of the signals in the system. We denote by $\Rbf_{\abf\bbf}$ the transfer matrix block from signal $\bbf$ to $\abf$ as shown in~\fig{R-S-realization}, and hence given a signal $\abf$, we have $\abf = \sum\limits_{\bbf} \Rbf_{\abf\bbf} \bbf + \dbf_{\abf}$, where $\dbf_{\abf}$ is the external disturbance on $\abf$. Notice that all dimensions in the internal state $\etabf$ have their corresponding share in $\dbf$, thereby avoiding the partial selection issues discussed in \cite{zheng2020equivalence}.

On the other hand, if we deem the external disturbance $\dbf$ as the input and the internal state $\etabf$ as the output, we can treat the closed-loop system as an open-loop system. We denote by the \emph{internal stability matrix} (or \emph{stability matrix} for short) $\Sbf$ the transfer matrix of such an open-loop system:
\begin{align}
\etabf = \Sbf \dbf.
\label{eqn:S}
\end{align}
We define $\Sbf_{\abf\bbf}$ as the transfer matrix block from disturbance on $\bbf$ to the signal $\abf$, and the columns in $\Sbf$ corresponding to $\bbf$ is denoted by $\Sbf_{:,\bbf}$.

The realization matrix $\Rbf$ and the stability matrix $\Sbf$ are related by the following lemma.
\begin{lemma}[Realization-Stability]\label{lem:R-S}
Let $\Rbf$ be the realization matrix and $\Sbf$ be the internal stability matrix, we have
\begin{align*}
(I - \Rbf) \Sbf = \Sbf(I - \Rbf) = I.
\end{align*}
\end{lemma}

\begin{proof}
Substituting \eqn{S} into \eqn{R} yields
\begin{align*}
(I - \Rbf)\etabf = (I - \Rbf)\Sbf \dbf = \dbf.
\end{align*}
Since $\dbf$ is arbitrary, we have
\begin{align*}
(I - \Rbf)\Sbf = I.
\end{align*}
Given $I - \Rbf$ and $\Sbf$ are both square matrices, we have
\begin{align*}
\Sbf = (I-\Rbf)^{-1}  \quad \Rightarrow \quad
\Sbf (I - \Rbf) = I,
\end{align*}
which concludes the proof.
\end{proof}

We remark that \lem{R-S} does not guarantee the existence of either $\Rbf$ or $\Sbf$. Rather, it says if both $\Rbf$ and $\Sbf$ exist, they must obey the relation. When they both exist, a consequence of \lem{R-S} is that $\Rbf \to \Sbf$ is a bijection map. In other words, if two systems have the same realization $\Rbf$ (or $I-\Rbf$, equivalently), they have the same internal stability $\Sbf$.

\subsection{Disturbance Transformation and Equivalent System}\label{sec:R-S_lemma-T}
In \eqn{R}, the external disturbance $\dbf$ affects each signal in the system independently. We can extend \eqn{R} and \eqn{S} to the cases where the dimensions in $\dbf$ are correlated. In particular, the external disturbance could be a \emph{transformation} $\Tbf$ on a different basis $\wbf$:
\begin{align*}
\dbf = \Tbf \wbf.
\end{align*}

When the transformation $\Tbf$ is invertible, we have
\begin{align*}
(I - \Rbf) \etabf = \Tbf \wbf
\quad\Rightarrow&\quad
\Tbf^{-1}(I - \Rbf) \etabf = \wbf = (I-\Rbf_{eq}) \etabf,\\
&\etabf = \Sbf \Tbf \wbf = \Sbf_{eq} \wbf.
\end{align*}
In other words, the transformation of the disturbance $\dbf = \Tbf \wbf$ can be seen as the derivation of an equivalent closed-loop system with realization $\Rbf_{eq}$ and stability $\Sbf_{eq}$ based on internal state $\etabf$ and external disturbance $\wbf$.

The derivation of an equivalent system is helpful for stability analysis. Since \lem{R-S} suggests that there is a bijection map from $\Rbf$ to $\Sbf$. If there are two systems with realizations $\Rbf_1$ and $\Rbf_2$ and we can relate them through an (invertible) transformation $\Tbf$ by
\begin{align*}
(I - \Rbf_2) = \Tbf^{-1}(I-\Rbf_1),
\end{align*}
their stability matrices will follow
\begin{align*}
\Sbf_2 = \Sbf_1 \Tbf.
\end{align*}

\subsection{Controller Synthesis and Column Dependency}\label{sec:R-S_lemma-controller-synthesis}
Notice that \lem{R-S} holds for arbitrary realization/internal stability matrices, e.g., non-causal $\Rbf$ and unstable $\Sbf$. When synthesizing a controller, we require the closed-loop system to be causal and internally stable. In other words, the transfer functions from one signal to any different signal should be proper, and the transfer functions from the external disturbance $\dbf$ to the internal state $\etabf$ should be stable proper, which are written as the following conditions:
\begin{align}
\Rbf_{\abf\bbf} \in \Rp, \forall \abf \neq \bbf, \quad\quad \Sbf \in \RHinf.
\label{eqn:R-S-condition}
\end{align}
Here, we implicitly require the existence of both $\Rbf$ and $\Sbf$. Accordingly, general controller synthesis problems (i.e., \emph{all} possible controller synthesis problems for a system described by some $\Rbf$) can be formulated as
\OptMinN{
g(\Rbf,\Sbf)
}{
\OptCons{(I-\Rbf)\Sbf = \Sbf(I-\Rbf) = I}{}{}\\
\OptCons{\Rbf_{\abf\bbf} \in \Rp}{\forall \abf \neq \bbf}{}\\
\OptCons{\Sbf \in \RHinf}{}{}\\
\OptCons{(\Rbf,\Sbf) \in \Ccal}{}{}
}
where $g$ is the objective function and $\Ccal$ represents the additional constraints on the realization and internal stability.
In the following sections, we will show that the existing controller synthesis methods/realization studies that focus on internal stability are essentially special cases of the feasible set in this general formulation.

A key constraint in the general controller synthesis problem is to enforce $\Sbf \in \RHinf$. Although we need to enforce all elements in $\Sbf$ to be in $\RHinf$, we can leverage the linear dependency among the components brought by \lem{R-S} to derive some parts automatically without explicit enforcement. In particular, we have \lem{dependency}.

\begin{lemma}\label{lem:dependency}
Let $\abf$ be a signal and $\Rbf_{\abf\abf} = O$, then
\begin{align*}
\Sbf_{:,\abf} = e_{\abf} + \sum\limits_{\bbf \neq \abf} \Sbf_{:,\bbf}\Rbf_{\bbf\abf}.
\end{align*}
\end{lemma}

\begin{proof}
By \lem{R-S}, we have $\Sbf (I -\Rbf) = I$ and hence
\begin{align*}
\Sbf_{:,\abf}(I - \Rbf_{\abf\abf}) - \sum\limits_{\bbf \neq \abf} \Sbf_{:,\bbf}\Rbf_{\bbf\abf} = e_{\abf}.
\end{align*}
The lemma follows as $\Rbf_{\abf\abf} = O$.
\end{proof}

\lem{dependency} can greatly reduce the decision variables when synthesizing a controller. For instance, the synthesized control $\ubf$ is usually a function of other signals except for itself, which implies $\Rbf_{\uu} = O$. Therefore, \lem{dependency} gives
\begin{align}
\Sbf_{:,\ubf} = e_{\ubf} + \sum\limits_{\bbf \neq \ubf} \Sbf_{:,\bbf}\Rbf_{\bbf\ubf}.
\label{eqn:u-column}
\end{align}

\section{Robust Stability Analysis}\label{sec:robust_analysis}
In this section, we derive the condition for robust stability analysis using \lem{R-S}. The condition then allows us to formulate the general robust controller synthesis problem.

\begin{figure}
\centering
\subcaptionbox{Realization perturbed by $\Deltabf$.\label{subfig:perturbed-realization-general-perturbation}}{\includegraphics{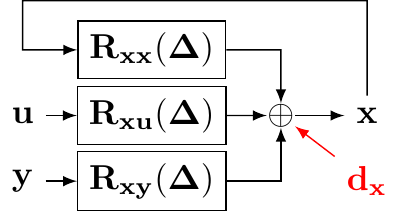}}\hfill
\subcaptionbox{Additive Perturbation.\label{subfig:perturbed-realization-additive-perturbation}}{\includegraphics{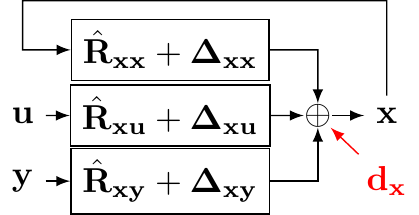}}
\caption{We denote by $\Rtru$ the realization matrix under perturbation $\Deltabf$ and by $\Rnm$ the nominal realization matrix without perturbation.}
\label{fig:perturbed-realization}
\end{figure}

\subsection{Robust Stability and Additive Perturbations}
Consider a system perturbed according to some uncertain parameter $\Deltabf \in \Dcal$ as in \subfig{perturbed-realization}{general-perturbation}, where $\Dcal$ is the uncertainty set. Denote by $\Rtru$ its realization matrix and by $\Stru$ the corresponding stability matrix.
By \lem{R-S}, the perturbed realization and stability matrices satisfy
\begin{align*}
(I - \Rtru) \Stru = \Stru (I - \Rtru) = I.
\end{align*}
Also, the perturbed system is \emph{robustly stable} if and only if the open-loop system from the external disturbance $\dbf$ to the internal state $\etabf$ is stable under all uncertain parameter $\Deltabf \in \Dcal$. In other words, we require the stability matrix $\Stru$ to obey
\begin{align}
\Stru \in \RHinf, \quad \forall \Deltabf \in \Dcal.
\label{eqn:robust-S-general}
\end{align}

Suppose the system is subject to additive perturbation\footnote{Some papers refer the additive perturbation here as ``multiplicative fault'' \cite{niemann2002reliable} since it appears in the equations as a multiplier of a signal.} and its realization matrix $\Rtru$ can be expressed as
\begin{align}
\Rtru = \Rnm + \Deltabf
\label{eqn:perturbed-R}
\end{align}
where $\Rnm$ is the \emph{nominal realization}, as shown in \subfig{perturbed-realization}{additive-perturbation}. Define the \emph{nominal stability} $\Snm$ as the stability matrix accompanying the nominal realization $\Rnm$, we can express $\Stru$ in terms of $\Snm$ and the perturbation $\Deltabf$ as follows.

\begin{lemma}[Stability under Additive Perturbation]\label{lem:perturbed-S}
Let $\Rnm$ be the nominal realization matrix and $\Snm$ be the nominal internal stability matrix. Suppose the system realization $\Rtru$ is subject to additive perturbation \eqn{perturbed-R},
the corresponding stability $\Stru$ is given by
\begin{align*}
\Stru = \Snm (I - \Deltabf \Snm)^{-1} = (I - \Snm\Deltabf)^{-1} \Snm.
\end{align*}
\end{lemma}

\begin{proof}
\lem{R-S} implies $(I - \Rnm) \Snm = I$.
Therefore,
\begin{align*}
(I - \Rtru) \Snm = (I - \Rnm - \Deltabf) \Snm = I - \Deltabf \Snm.
\end{align*}
As a result, \lem{R-S} suggests
\begin{gather*}
(I - \Rtru) \Snm (I - \Deltabf \Snm)^{-1} = I \nonumber \\
\Rightarrow\quad
\Stru = \Snm (I - \Deltabf \Snm)^{-1}.
\end{gather*}
Similarly, we can also derive $\Stru = (I - \Snm \Deltabf)^{-1} \Snm$ from $\Snm(I-\Rnm) = I$, and the lemma follows.
\end{proof}

We can interpret the resulting stability matrix $\Sbf(\Deltabf)$ in \lem{perturbed-S} as the nominal stability $\Snm$ with a feedback path $\Deltabf$ as in \fig{perturbed-stability}. From this perspective, an additive perturbation to the realization results in a feedback path in the stability.

\begin{figure}
\centering
\includegraphics{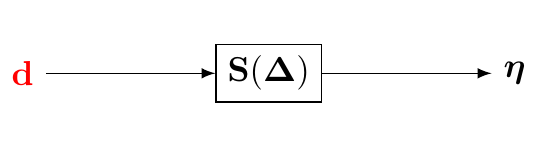}
\includegraphics{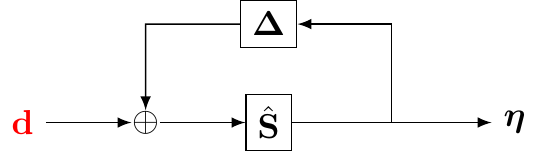}
\caption{The stability matrix $\Sbf$ maps external disturbance $\dbf$ to internal state $\etabf$. We denote by $\Stru$ the stability matrix under perturbation $\Deltabf$ and by $\Snm$ the nominal stability matrix without perturbation. When the perturbation is additive to the realization, it results in a feedback path for the stability.}
\label{fig:perturbed-stability}
\end{figure}

According to \lem{perturbed-S} and condition \eqn{robust-S-general}, to ensure robust stability of the system, we need
\begin{align}
\Snm (I - \Deltabf \Snm)^{-1} \in \RHinf, \quad \forall \Deltabf \in \Dcal.
\label{eqn:robust-S}
\end{align}

\subsection{General Formulation for Robust Controller Synthesis}
We can now generalize the general controller synthesis problem in \cite{tseng2021realization} to its robust version using condition \eqn{robust-S-general}:
\OptMinN{
g(\Rtru,\Stru,\Dcal)
}{
&\ (I-\Rtru)\Stru = \Stru&&(I-\Rtru) = I \\
\OptCons{}{\forall \Deltabf \in \Dcal}{}\\
\OptCons{\Rtru_{\abf\bbf} \in \Rp}{\forall \Deltabf \in \Dcal, \abf \neq \bbf}{}\\
\OptCons{\Stru \in \RHinf}{\forall \Deltabf \in \Dcal}{}\\
\OptCons{(\Rtru,\Stru) \in \Ccal}{\forall \Deltabf \in \Dcal}{}
}
where $\Ccal$ represents some additional constraints on the realization and stability. In particular, for a system subject to additive perturbation, the problem can be reformulated as
\OptMinN{
g(\Rnm,\Snm,\Dcal)
}{
\OptCons{(I-\Rnm)\Snm = \Snm(I-\Rnm) = I}{}{}\\
\OptCons{\Rtru_{\abf\bbf} = (\Rnm + \Deltabf)_{\abf\bbf} \in \Rp}{\forall \Deltabf \in \Dcal, \abf \neq \bbf}{}\\
\OptCons{\Stru = \Snm (I - \Deltabf \Snm)^{-1} \in \RHinf}{\forall \Deltabf \in \Dcal}{}\\
\OptCons{(\Rtru,\Stru) \in \Ccal}{\forall \Deltabf \in \Dcal}{}
}
This formulation is general as it can describe not only the robust controller synthesis problem for a given uncertainty set $\Dcal$ but also the stability margin problem like $\mu$-synthesis \cite{doyle1982analysis, doyle1985structured, zhou1998essentials}, where $\Dcal$ is itself a variable to ``maximize''.

Despite its generality, solving this general formulation is challenging in general, and the major obstacle is to ensure the perturbed stability for all $\Deltabf \in \Dcal$. Except for some computationally tractable cases as those listed in \sec{existing_results}, enforcing the robust constraints involves dealing with semi-infinite programming when the $\Dcal$ has infinite cardinality. For those cases, one may instead enforce chance constraints and adopt sampling-based techniques as in \cite{calafiore2005uncertain, erdougan2006ambiguous,tseng2016random}.

\section{Corollaries: Controller Synthesis}\label{sec:controller_synthesis}
We use \lem{R-S} and condition \eqn{R-S-condition} to derive existing controller synthesis proposals, including Youla \cite{youla1976modern2}, input-output \cite{furieri2019input}, system level \cite{wang2019system,anderson2019system}, mixed parameterizations \cite{zheng2019system}, and generalized system level synthesis \cite{szabo2021generalised}, with different $\Rbf$ and $\Sbf$ structures. We then demonstrate a simpler way to obtain the results in \cite{zheng2020equivalence} using transformations.

\begin{figure}
\centering
\includegraphics[scale=1]{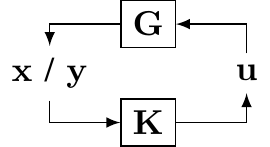}
\caption{The realization with plant $\Gbf$ and controller $\Kbf$. The internal signals include state $\xbf$ (or measurement $\ybf$) and control $\ubf$.}
\label{fig:realization-G-K}
\end{figure}

\subsection{Youla Parametrization}\label{sec:controller_synthesis-Youla}
Youla parameterization is based on the doubly coprime factorization of the plant $\Gbf$. If $\Gbf$ is stabilizable and detectable, we have
\begin{align}
\mat{
\Mbf_l & -\Nbf_l\\
-\Vbf_l & \Ubf_l
}\mat{
\Ubf_r & \Nbf_r\\
\Vbf_r & \Mbf_r
} = I
\label{eqn:coprime-factorization}
\end{align}
where both matrices are in $\RHinf$, $\Mbf_l$ and $\Mbf_r$ are both invertible in $\RHinf$, and $\Gbf = \Mbf_l^{-1} \Nbf_l = \Nbf_r \Mbf_r^{-1}$ \cite[Theorem 5.6]{zhou1998essentials}.

The following corollary is a modern rewrite of the original Youla parameterization in \cite[Lemma 3]{youla1976modern2} given by \cite[Theorem 11.6]{zhou1998essentials}:
\begin{corollary}
Let the plant $\Gbf$ be doubly coprime factorizable. Given $\Qbf \in \RHinf$, the set of all proper controllers achieving internal stability is parameterized by
\begin{align*}
\Kbf = (\Vbf_r - \Mbf_r \Qbf)(\Ubf_r - \Nbf_r \Qbf)^{-1}.
\end{align*}
\end{corollary}

\begin{proof}
Consider the realization in \fig{realization-G-K}, which has
\begin{align*}
\Rbf = \mat{O & \Gbf \\ \Kbf & O}, \quad \etabf = \mat{\xbf\\ \ubf}.
\end{align*}

To show that all $\Kbf$ can be parameterized by $\Qbf \in \RHinf$, we need to show that each $\Qbf$ is mapped to one valid $\Kbf$ and vice versa. For mapping $\Qbf$ to $\Kbf$, we consider the following transformation:
\begin{align*}
\Tbf^{-1} =&
\mat{
\Mbf_l^{-1} & O\\
O & (\Ubf_l -\Qbf\Nbf_l)^{-1}\\
}
\mat{
I & O\\
\Qbf & I
}\\
\Tbf =&
\mat{
I & O\\
-\Qbf & I
}
\mat{
\Mbf_l & O\\
O & \Ubf_l -\Qbf\Nbf_l \\
}
\end{align*}
As such,
\begin{align*}
I =&\ \Tbf^{-1} \mat{
\Mbf_l & -\Nbf_l\\
-\Vbf_l & \Ubf_l
}\mat{
\Ubf_r & \Nbf_r\\
\Vbf_r & \Mbf_r
}\Tbf\\
=&
\mat{I & -\Gbf \\ -\Kbf & I}
\mat{
(\Ubf_r - \Nbf_r \Qbf) \Mbf_l & \Sbf_{\yu}\\
(\Vbf_r - \Mbf_r \Qbf) \Mbf_l & \Sbf_{\uu}
} = (I -\Rbf)\Sbf
\end{align*}
where $\Sbf_{\yu}$ and $\Sbf_{\uu}$ are given by \eqn{u-column} and $\Gbf = \Mbf_l^{-1}\Nbf_l$:
\begin{align*}
\mat{\Sbf_{\ybf\ubf} \\ \Sbf_{\ubf\ubf}} =&
\mat{O\\ I} +
\( \mat{\Ubf_r \\ \Vbf_r} -  \mat{\Nbf_r \\ \Mbf_r} \Qbf \) \Mbf_l \Gbf\\
=&
\mat{O\\ I} +
\( \mat{\Ubf_r \\ \Vbf_r} -  \mat{\Nbf_r \\ \Mbf_r} \Qbf \) \Nbf_l.
\end{align*}

Since $\Tbf \in \RHinf$, we have $\Sbf \in \RHinf$. Therefore,
\begin{align*}
-\Kbf (\Ubf_r - \Nbf_r \Qbf) + (\Vbf_r - \Mbf_r \Qbf) = O,
\end{align*}
which leads to the desired $\Kbf$.

On the other hand, for mapping $\Kbf$ to $\Qbf$, internal stability of $\Kbf$ implies the corresponding $\Sbf_{\ux} \in \RHinf$. We compute $\Qbf$ by
\begin{align*}
\Qbf = \Mbf_r^{-1} (\Vbf_r - \Sbf_{\ux}\Mbf_l^{-1}),
\end{align*}
which is also in $\RHinf$ as $\Mbf_l$ and $\Mbf_r$ are both invertible in $\RHinf$ (i.e., $\Mbf_l^{-1}, \Mbf_r^{-1} \in \RHinf$), and all elements in $\Sbf$ can be expressed in $\Qbf$ using \lem{R-S}.
\end{proof}

\def\IOPqple{\{\Ybf,\Ubf,\Wbf,\Zbf\}}
\subsection{Input-Output Parametrization (IOP)}\label{sec:controller_synthesis-IOP}
Inspired by the system level approach in \cite{wang2019system}, \cite{furieri2019input} revisits the input-output system studied by Youla parameterization and proposes IOP as follows that does not depend on the doubly coprime factorization \cite[Theorem 1]{furieri2019input}.

\begin{corollary}\label{cor:IOP}
For the realization in \fig{realization-G-K} with $\Gbf \in \Rsp$, the set of all proper internally stabilizing controller is parameterized by $\IOPqple$ that lies in the affine subspace defined by the equations
\begin{align*}
\mat{I & -\Gbf}
\mat{\Ybf & \Wbf\\ \Ubf & \Zbf}
=&
\mat{I & O},\\
\mat{\Ybf & \Wbf\\ \Ubf & \Zbf}
\mat{-\Gbf\\ I}
=&
\mat{O\\I},\\
\Ybf, \Ubf, \Wbf, \Zbf \in&\ \RHinf,
\end{align*}
and the controller is given by $\Kbf = \Ubf \Ybf^{-1}$.
\end{corollary}

\begin{proof}
We can write down the realization matrix in~\fig{realization-G-K}:
\begin{align*}
\Rbf = \mat{O & \Gbf \\ \Kbf & O}, \quad \etabf = \mat{\ybf \\ \ubf}.
\end{align*}

Given $\Kbf$, the derivation of $\IOPqple$ is a direct consequence of \lem{R-S} and condition \eqn{R-S-condition}, which suggest
\begin{align}
I =&\ (I-\Rbf)\Sbf = \mat{I & -\Gbf \\ -\Kbf & I}
\mat{\Ybf & \Wbf\\ \Ubf & \Zbf} \label{eqn:IOP}\\
=&\ \Sbf(I-\Rbf) =
\mat{\Ybf & \Wbf\\ \Ubf & \Zbf}\mat{I & -\Gbf \\ -\Kbf & I}, \nonumber\\
&\ \Ybf, \Ubf, \Wbf, \Zbf \in \RHinf. \nonumber
\end{align}

Conversely, given $\IOPqple$, \eqn{IOP} implies
\begin{align*}
\Ubf = \Kbf \Ybf \quad\Rightarrow\quad
\Kbf = \Ubf \Ybf^{-1}.
\end{align*}
We need to verify that $\Ybf$ is invertible in $\Rp$ so that $\Kbf \in \Rp$. Given $\Gbf \in \Rsp$, we know that
\begin{align*}
\Ybf = I + \Gbf \Ubf = I + (zI-\Lambda)^{-1} \Jbf.
\end{align*}
for some matrix $\Lambda$ and $\Jbf \in \Rp$. As a result,
\begin{align*}
\Ybf^{-1} = I + \sum\limits_{k\geq 1}^{\infty} (zI-\Lambda)^{-k}\Jbf^k \in \Rp,
\end{align*}
which concludes the proof.
\end{proof}

\cite{szabo2021generalised} also provides an alternative proof to the corollary that differs from \cite{furieri2019input} and the simple proof here.

\subsection{System Level Parametrization/Synthesis (SLP/SLS)}\label{sec:controller_synthesis-SLS}
System level synthesis (SLS) uses system level parameterization (SLP) to parameterize internally stabilizing controllers. There are two SLPs: for state-feedback and output-feedback systems, respectively. We discuss them below.

\begin{figure}
\centering
\includegraphics[scale=1]{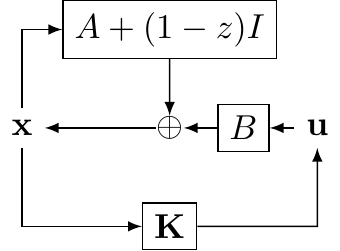}
\caption{The realization of a state-feedback system with controller $\Kbf$. The internal signals are state $\xbf$ and control $\ubf$.}
\label{fig:realization-state-feedback}
\end{figure}

\renewcommand{\paragraph}[1]{\vspace*{0.5\baselineskip}\noindent\textbf{#1}}

\paragraph{State-Feedback:} The following state-feedback parameterization is given in \cite[Theorem 1]{wang2019system} and \cite[Theorem 4.1]{anderson2019system}.

\begin{corollary}\label{cor:sf-SLS}
For the realization in~\fig{realization-state-feedback}, the set of all proper internally stabilizing state-feedback controller is parameterized by $\SFpair$ that lies in the affine space defined by
\begin{align*}
\mat{zI-A & -B}\mat{\Phibf_{\xbf}\\ \Phibf_{\ubf}} = I,\\
\Phibf_{\xbf}, \Phibf_{\ubf}\in z^{-1}\RHinf,
\end{align*}
and the controller is given by $\Kbf = \Phibf_{\ubf} \Phibf_{\xbf}^{-1}$.
\end{corollary}

\begin{proof}
The realization matrix in~\fig{realization-state-feedback} is
\begin{align}
\Rbf = \mat{
A + (1-z)I & B \\
\Kbf & O
}, \quad \etabf = \mat{\xbf\\ \ubf}.
\label{eqn:sf-SLS-R}
\end{align}

To derive $\Kbf$ from $\SFpair$, \lem{R-S} and condition \eqn{R-S-condition} lead to
\begin{align*}
\mat{
zI-A & -B \\
-\Kbf & I}&
\mat{\Phibf_{\xbf} & \Sbf_{\xu} \\
\Phibf_{\ubf} & \Sbf_{\uu}}
= I,\\
\Kbf \in \Rp,\quad &\Phibf_{\xbf}, \Phibf_{\ubf} \in \RHinf.
\end{align*}
Meanwhile, since
\begin{align*}
(zI-A) \Phibf_{\xbf} = I + B \Phibf_{\ubf} \in \RHinf,
\end{align*}
we have $\Phibf_{\xbf} \in z^{-1}\RHinf$. As a result, given $\Kbf \in \Rp$,
\begin{align*}
\Phibf_{\ubf} = \Kbf \Phibf_{\xbf} = z^{-1} \Kbf (z\Phibf_{\xbf}) \in z^{-1}\Rp,
\end{align*}
we know $\Phibf_{\ubf} \in z^{-1}\Rp \cap \RHinf = z^{-1}\RHinf$.

Conversely, given $\SFpair$, we can derive $\Kbf = \Phibf_{\ubf} \Phibf_{\xbf}^{-1}$ from \lem{R-S}. It remains to show that $\Sbf_{\xu}$ and $\Sbf_{\uu}$ exist whenever $\SFpair$ is given. According to \eqn{u-column}
\begin{align*}
\mat{\Sbf_{\xu}\\ \Sbf_{\uu}} =
\mat{O\\I} + \mat{\Phibf_{\xbf}\\ \Phibf_{\ubf}}B \in \RHinf,
\end{align*}
which concludes the proof.
\end{proof}

\begin{figure}
\centering
\includegraphics[scale=1]{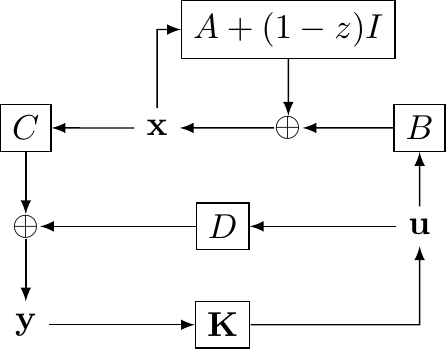}
\caption{The realization of an output-feedback system with controller $\Kbf$. The internal state $\etabf$ consists of state $\xbf$, control $\ubf$, and measurement $\ybf$ signals.}
\label{fig:realization-output-feedback}
\end{figure}

\paragraph{Output-Feedback:} The output-feedback SLP below is from \cite[Theorem 2]{wang2019system} and \cite[Theorem 5.1]{anderson2019system}.

\begin{corollary}\label{cor:of-SLS-zero-D}
For the realization in~\fig{realization-output-feedback} with $D = O$, the set of all proper internally stabilizing output-feedback controller is parameterized by $\OFqple$ that lies in the affine space defined by
\begin{subequations}\label{eqn:of-SLS-constraints}
\begin{align}
\mat{
zI - A & -B
}
\mat{
\Phibf_{\xx} & \Phibf_{\xy}\\
\Phibf_{\ux} & \Phibf_{\uy}
}
=&
\mat{I & O},
\label{eqn:of-SLS-constraints-1}
\\
\mat{
\Phibf_{\xx} & \Phibf_{\xy}\\
\Phibf_{\ux} & \Phibf_{\uy}
}
\mat{
zI - A \\ -C
}
=&
\mat{I \\ O},\nonumber \\
\Phibf_{\xx}, \Phibf_{\xy}, \Phibf_{\ux} \in z^{-1}\RHinf,&\ \quad \Phibf_{\uy} \in \RHinf,
\label{eqn:of-SLS-constraints-3}
\end{align}
\end{subequations}
and the controller is given by
\begin{align*}
\Kbf = \Phibf_{\uy} - \Phibf_{\ux} \Phibf_{\xx}^{-1} \Phibf_{\xy}.
\end{align*}

\end{corollary}

In fact, we can extend \cor{of-SLS-zero-D} to general $D$.
\begin{corollary}\label{cor:of-SLS}
Given $\OFqple$ that lies in the affine space in \cor{of-SLS-zero-D} and an arbitrary $D$, the proper internally stabilizing output-feedback controller $\Kbf$ is given by
\begin{align*}
\Kbf = \Kbf_0 \left( I + D \Kbf_0 \right)^{-1}
\end{align*}
where
$\Kbf_0 = \Phibf_{\uy} - \Phibf_{\ux} \Phibf_{\xx}^{-1} \Phibf_{\xy}$.
\end{corollary}

We prove the more general version -- \cor{of-SLS} -- below.

\begin{proof}
The realization matrix in~\fig{realization-output-feedback} is
\begin{align*}
\Rbf = \mat{
A + (1-z)I & B & O \\
O & O & \Kbf \\
C & D & O
}, \quad \etabf = \mat{\xbf\\ \ubf\\ \ybf}.
\end{align*}

Given $\Kbf$ we can directly derive $\OFqple$ from \lem{R-S}
\begin{align}
I =& \mat{
zI-A & -B & O \\
O & I & -\Kbf \\
-C & -D & I
}
\mat{
\Phibf_{\xx} & \Sbf_{\xu} & \Phibf_{\xy} \\
\Phibf_{\ux} & \Sbf_{\uu} & \Phibf_{\uy} \\
\Sbf_{\yx} & \Sbf_{\yu} & \Sbf_{\yy}
} \label{eqn:of-SLS}\\
=& \mat{
\Phibf_{\xx} & \Sbf_{\xu} & \Phibf_{\xy} \\
\Phibf_{\ux} & \Sbf_{\uu} & \Phibf_{\uy} \\
\Sbf_{\yx} & \Sbf_{\yu} & \Sbf_{\yy}
}
\mat{
zI-A & -B & O \\
O & I & -\Kbf \\
-C & -D & I
},\nonumber
\end{align}
where $\Sbf \in \RHinf$ by condition \eqn{R-S-condition}. As a result, we have
\begin{align*}
(zI-A)\Phibf_{\xx} =&\ I + B\Phibf_{\ux} \in \RHinf,\\
(zI-A)\Phibf_{\xy} =&\ B\Phibf_{\uy} \in \RHinf,\\
\Phibf_{\ux}(zI-A) =&\ \Phibf_{\uy}C \in \RHinf.
\end{align*}
Therefore, $\Phibf_{\xx},\Phibf_{\ux},\Phibf_{\xy} \in z^{-1}\RHinf$ and $\Phibf_{\uy} \in \RHinf$. 

Conversely, we can derive $\Kbf$ from $\OFqple$ as follows. First, we multiply the matrix
\begin{align*}
\Gamma = \mat{
I & B & O\\
O & I & O\\
O & D & I
} 
\end{align*}
at the left of both sides of \eqn{of-SLS}, which leads to
\begin{align*}
\mat{
zI-A & -B \Kbf \\
-C & I - D \Kbf
}
\mat{
\Phibf_{\xx} & \Phibf_{\xy} \\
\Sbf_{\yx} & \Sbf_{\yy}
} = I.
\end{align*}
Therefore, as $\Phibf_{\xx}$ and $\Sbf_{\yy}$ are both square, taking matrix inverse, we have
\begin{align*}
I - D \Kbf = \left( \Sbf_{\yy} -  \Sbf_{\yx} \Phibf_{\xx}^{-1} \Phibf_{\xy} \right)^{-1}.
\end{align*}
Since $\Phibf_{\ux} = \Kbf \Sbf_{\yx}$ and $\Phibf_{\uy} = \Kbf \Sbf_{\yy}$,
we know
\begin{align*}
\Kbf_0 = \Kbf \left( \Sbf_{\yy} -  \Sbf_{\yx} \Phibf_{\xx}^{-1} \Phibf_{\xy} \right)
\end{align*}
and we can rearrange the equation to obtain
\begin{align*}
\Kbf_0 - \Kbf D \Kbf_0 = \Kbf
\quad \Rightarrow \quad
\Kbf = \Kbf_0 \left( I + D \Kbf_0 \right)^{-1}.
\end{align*}

The last thing we need to verify is that $\Sbf$ exists and is in $\RHinf$. By \lem{R-S}, we know
\begin{align}
\Sbf_{\yx} =&\ C \Phibf_{\xx} + D \Phibf_{\ux} \in \RHinf, \nonumber \\
\Sbf_{\yy} =&\ C \Phibf_{\xy} + D \Phibf_{\uy} + I \in \RHinf. 
\label{eqn:of-SLS-S-1}
\end{align}
and we can compute the rest by \eqn{u-column}
\begin{align}
\mat{\Sbf_{\xu}\\ \Sbf_{\uu}\\ \Sbf_{\yu}} = 
\mat{O\\I\\O} + 
\mat{\Phibf_{\xx}\\ \Phibf_{\ux}\\ \Sbf_{\yx}}B +
\mat{\Phibf_{\xy}\\ \Phibf_{\uy}\\ \Sbf_{\yy}}D \in \RHinf,
\label{eqn:of-SLS-S-2}
\end{align}
which concludes the proof.
\end{proof}

\subsection{Mixed Parameterizations}
Letting $\Gbf = C(zI-A)^{-1}B + D$, \cite[Proposition 3, Proposition 4]{zheng2019system} provides the following corollaries that have conditions in both SLP and IOP flavors.

\begin{corollary}\label{cor:mixed-1}
For the realization in~\fig{realization-output-feedback}, the set of all proper internally stabilizing output-feedback controller is parameterized by $\{\Phibf_{\yx},\Phibf_{\ux},\Phibf_{\yy},\Phibf_{\uy}\}$ that lies in the affine space defined by
\begin{align*}
\mat{
I & -\Gbf
}
\mat{
\Phibf_{\yx} & \Phibf_{\yy}\\
\Phibf_{\ux} & \Phibf_{\uy}
}
=&
\mat{C(zI-A)^{-1} & I},\\
\mat{
\Phibf_{\yx} & \Phibf_{\yy}\\
\Phibf_{\ux} & \Phibf_{\uy}
}
\mat{
zI - A \\ -C
}
=&\
O,\\
\Phibf_{\yx},\Phibf_{\ux},\Phibf_{\yy}&,\Phibf_{\uy} \in \RHinf,
\end{align*}
and the controller is given by
\begin{align*}
\Kbf = \Phibf_{\uy}\Phibf_{\yy}^{-1}.
\end{align*}
\end{corollary}

\begin{corollary}\label{cor:mixed-2}
For the realization in~\fig{realization-output-feedback}, the set of all proper internally stabilizing output-feedback controller is parameterized by $\{\Phibf_{\xy},\Phibf_{\uy},\Phibf_{\xu},\Phibf_{\uu}\}$ that lies in the affine space defined by
\begin{align*}
\mat{
zI-A & -B
}
\mat{
\Phibf_{\xy} & \Phibf_{\xu}\\
\Phibf_{\uy} & \Phibf_{\uu}
}
=&\
O,\\
\mat{
\Phibf_{\yx} & \Phibf_{\yy}\\
\Phibf_{\ux} & \Phibf_{\uy}
}
\mat{
-\Gbf \\ I
}
=&
\mat{
(zI - A)^{-1}B \\ I
},\\
\Phibf_{\xy},\Phibf_{\uy},\Phibf_{\xu},\Phibf_{\uu} \in&\ \RHinf,
\end{align*}
and the controller is given by
\begin{align*}
\Kbf = \Phibf_{\uu}^{-1}\Phibf_{\uy}.
\end{align*}
\end{corollary}

We give a brief proof below for the two corollaries above.

\begin{proof}

\lem{R-S} gives
\begin{align*}
I =&\ (I-\Rbf)\Sbf\\
=&
\mat{
zI-A & -B & O \\
O & I & -\Kbf \\
-C & -D & I
}
\mat{
\Sbf_{\xx} & \Phibf_{\xu} & \Phibf_{\xy} \\
\Phibf_{\ux} & \Phibf_{\uu} & \Phibf_{\uy} \\
\Phibf_{\yx} & \Sbf_{\yu} & \Phibf_{\yy}
}.
\end{align*}
We consider two matrices
\begin{align*}
\Gamma_1 =& \mat{
I & O & O\\
O & I & O\\
C(zI-A)^{-1} & O & I
},\\
\Gamma_2 =& \mat{
I & (zI-A)^{-1}B & O\\
O & I & O\\
O & O & I
}.
\end{align*}

Analogous to the proof of \cor{of-SLS}, \cor{mixed-1} can be derived from the following conditions and condition \eqn{R-S-condition}.
\begin{align*}
\Gamma_1(I-\Rbf)\Sbf = \Gamma_1, \quad
\Sbf(I-\Rbf) = I.
\end{align*}
Similarly, we derive \cor{mixed-2} from condition \eqn{R-S-condition} and
\begin{align*}
(I-\Rbf)\Sbf = I, \quad
\Sbf(I-\Rbf)\Gamma_2 = \Gamma_2.
\end{align*}
\end{proof}

\subsection{Generalized System Level Synthesis}

\begin{figure}
\centering
\includegraphics[scale=1]{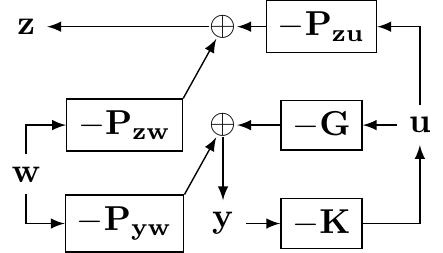}
\caption{In addition to the measurement $\ybf$ and the control $\ubf$, \cite{szabo2021generalised} generalizes the realization in \fig{realization-G-K} to consider also the stability of output $\zbf$ under an additional external disturbance $\wbf$. Notice that the plant $\Gbf$ and the controller $\Kbf$ are negative in this setting.}
\label{fig:realization-Szabo2021}
\end{figure}

\fig{realization-G-K} consists of the state/measurement $\xbf$/$\ybf$ and the control $\ubf$. \cite{szabo2021generalised} generalizes the setting to study how a controller could also stabilize some output $\zbf$ under an additional external disturbance $\wbf$ as shown in \fig{realization-Szabo2021}. The following corollary is a slight rewrite of the proposed generalized system level synthesis in \cite[Theorem 2]{szabo2021generalised}.\footnote{Although bearing the name ``generalized system level synthesis,'' \cor{Szabo2021} actually generalizes input-output parameterization (\cor{IOP}).}

\begin{corollary}\label{cor:Szabo2021}
For the realization in \fig{realization-Szabo2021} with $\Gbf \in \Rsp$, the set of all proper internally stabilizing controller is parameterized by $\Psibf$ where
\begin{align*}
\Psibf = \mat{
\Psibf_{\yy} & \Psibf_{\yu} & \Psibf_{\yw} \\
\Psibf_{\uy} & \Psibf_{\uu} & \Psibf_{\uw} \\
\Psibf_{\zy} & \Psibf_{\zu} & \Psibf_{\zw}
} \in \RHinf
\end{align*}
and $\Psibf$ lies in the affine space defined by
\begin{align*}
\mat{
& O\\
\Psibf & O \\
& I}
\mat{
\Gbf & \Pbf_{\yw}\\
I & O\\
O & I\\
\Pbf_{\zu} & \Pbf_{\zw}
}
=&
\mat{
O & O\\
I & O\\
O & O
},\\
\mat{
I & \Gbf & O & \Pbf_{\yw}\\
O & \Pbf_{\zu} & I & \Pbf_{\zw}
}
\mat{
& \Psibf& \\
O & O & I
}
=&
\mat{
I & O & O \\
O & O & O
}
\end{align*}
and the controller is given by $\ubf = -\Kbf \ybf$ where
\begin{align*}
\Kbf = -\Psibf_{\uy}\Psibf_{\yy}^{-1}.
\end{align*}
\end{corollary}

We prove \cor{Szabo2021} using \lem{R-S} as follows.

\begin{proof}
The realization matrix in~\fig{realization-Szabo2021} is
\begin{align*}
\Rbf = \mat{
O & -\Gbf & O & -\Pbf_{\yw} \\
-\Kbf & O & O & O \\
O & -\Pbf_{\zu} & O & -\Pbf_{\zw}\\
O & O & O & O
}, \quad \etabf = \mat{\ybf\\ \ubf\\ \zbf \\ \wbf}.
\end{align*}
\lem{R-S} implies
\begin{align*}
I =&\ (I - \Rbf) \Sbf \\
=& \mat{
I & \Gbf & O & \Pbf_{\yw} \\
\Kbf & I & O & O \\
O & \Pbf_{\zu} & I & \Pbf_{\zw}\\
O & O & O & I
}
\mat{
\Psibf_{\yy} & \Psibf_{\yu} & \Sbf_{\yz} & \Psibf_{\yw} \\
\Psibf_{\uy} & \Psibf_{\uu} & \Sbf_{\uz} & \Psibf_{\uw} \\
\Psibf_{\zy} & \Psibf_{\zu} & \Sbf_{\zz} & \Psibf_{\zw} \\
\Sbf_{\wy} & \Sbf_{\wu} & \Sbf_{\wz} & \Sbf_{\ww}
}.
\end{align*}
By the $3^{\rm rd}$ column and the $4^{\rm th}$ row of $I-\Rbf$, we know
\begin{align*}
\Sbf =
\mat{
\Psibf_{\yy} & \Psibf_{\yu} & O & \Psibf_{\yw} \\
\Psibf_{\uy} & \Psibf_{\uu} & O & \Psibf_{\uw} \\
\Psibf_{\zy} & \Psibf_{\zu} & I & \Psibf_{\zw} \\
O & O & O & I
}.
\end{align*}
Accordingly, we can take $\Sbf$ except for its last row multiplied by the $2^{\rm nd}$ and the $4^{\rm th}$ columns of $I-\Rbf$ to derive
\begin{align*}
\mat{
& O\\
\Psibf & O \\
& I}
\mat{
\Gbf & \Pbf_{\yw}\\
I & O\\
O & I\\
\Pbf_{\zu} & \Pbf_{\zw}
}
=&
\mat{
O & O\\
I & O\\
O & O
}
\end{align*}
after rearranging the columns of $\Sbf$.
Similarly, we can derive
\begin{align*}
\mat{
I & \Gbf & O & \Pbf_{\yw}\\
O & \Pbf_{\zu} & I & \Pbf_{\zw}
}
\mat{
& \Psibf& \\
O & O & I
}
=&
\mat{
I & O & O \\
O & O & O
}
\end{align*}
from multiplying the $1^{\rm st}$ and $3^{\rm rd}$ rows of $I-\Rbf$ by $\Sbf$ except for the $3^{\rm rd}$ column.

Lastly, we can derive $\Kbf$ from the $2^{\rm nd}$ row of $I - \Rbf$ multiplied by the $1^{\rm st}$ column of $\Sbf$, which yields
\begin{align*}
\Kbf \Psibf_{\yy} + \Psibf_{\uy} = O.
\end{align*}
Below we show that $\Psibf_{\yy}$ is invertible in $\Rp$, which allows us to derive $\Kbf$ as desired. Similar to the proof of \cor{IOP}, since $\Gbf \in \Rsp$ and
\begin{align*}
\Psibf_{\yy} + \Gbf \Psibf_{\uy} = I,
\end{align*}
we can write
\begin{align*}
\Psibf_{\yy} = I + (zI - \Lambda)^{-1} \Jbf
\end{align*}
for some matrix $\Lambda$ and $\Jbf \in \Rp$. As a result, $\Psibf_{\yy}$ is invertible in $\Rp$ as
\begin{align*}
\Psibf_{\yy}^{-1} = I + \sum\limits_{k \geq 1}^{\infty} (zI-\Lambda)^{-k} \Jbf^k \in \Rp.
\end{align*}

\end{proof}

Notice that if $\Psibf_{\uu}$ is invertible, we can also alternatively derive $\Kbf = - \Psibf_{\uu}^{-1} \Psibf_{\uy}$ from multiplying the $2^{\rm nd}$ row of $\Sbf$ by the $1^{\rm st}$ column of $I - \Rbf$.

\subsection{Equivalence among Synthesis Methods}

The parameterizations above are shown equivalent in \cite{zheng2020equivalence} through careful calculations. Here we demonstrate how \lem{R-S} and transformations lead to more straightforward derivations of equivalent components.

\lem{R-S} implies that $\Rbf \to \Sbf$ is a one-to-one mapping. Therefore, to show the equivalence among different synthesis methods, we can simply find a transformation $\Tbf$ such that the equivalent system has the same realization as the other system. As such, \lem{R-S} suggests that the stability matrices are the same, and we just need to compare the elements correspondingly.

When comparing a state-feedback system with an output-feedback system in the following analyses, we assume that the state $\xbf$ is taken as the measurement $\ybf$.

\paragraph{Youla parameterization and IOP:} Youla parameterization and IOP share the same realization in~\fig{realization-G-K} (except for changing $\xbf$ to $\ybf$). Therefore,
\begin{align*}
\mat{
\Ybf & \Wbf\\
\Ubf & \Zbf
}
=&
\mat{
\Ubf_r & \Nbf_r\\
\Vbf_r & \Mbf_r
}
\mat{
I & O\\
-\Qbf & I
}
\mat{
\Mbf_l & O\\
O & \Ubf_l -\Qbf\Nbf_l \\
}
\\
=&
\mat{
(\Ubf_r - \Nbf_r \Qbf) \Mbf_l & (\Ubf_r - \Nbf_r \Qbf) \Nbf_l \\
(\Vbf_r - \Mbf_r \Qbf) \Mbf_l & I + (\Vbf_r - \Mbf_r \Qbf) \Nbf_l
}.
\end{align*}

\paragraph{IOP and SLP:} We then show the equivalence between IOP and SLP. For state-feedback SLP with realization in~\fig{realization-state-feedback}, we perform the transformation
\begin{align*}
\Tbf^{-1} = \mat{(zI-A)^{-1} & O\\ O & I}, \quad
\Tbf = \mat{zI-A & O\\ O & I},
\end{align*}
which leads to
\begin{align*}
\Tbf^{-1} \mat{
zI-A & -B \\
-\Kbf & I} =
\mat{I & -\Gbf \\ -\Kbf & I}.
\end{align*}
Accordingly, the stability matrix becomes
\begin{align*}
\mat{\Ybf & \Wbf \\ \Ubf & \Zbf} =& \mat{\Phibf_{\xbf} & \Phibf_{\xbf}B \\ \Phibf_{\ubf} & I + \Phibf_{\ubf}B }\Tbf \\
=&
\mat{\Phibf_{\xbf}(zI-A) & \Phibf_{\xbf}B(zI-A) \\ \Phibf_{\ubf} & I + \Phibf_{\ubf}B }.
\end{align*}

For output-feedback SLP, we consider the transformation
\begin{align*}
\Tbf^{-1} =& \mat{
I & O & O\\
O & I & O\\
C(zI-A)^{-1} & O & I
},\\
\Tbf =& \mat{
I & O & O\\
O & I & O\\
-C(zI-A)^{-1} & O & I
},
\end{align*}
which leads to
\begin{align*}
\Tbf^{-1} (I - \Rbf) =&\ \Tbf^{-1}
\mat{
zI-A & -B & O \\
O & I & -\Kbf \\
-C & -D & I
}\\
=&
\mat{
zI-A & -B & O\\
O & I & -\Kbf \\
O & -\Gbf & I
}, \quad \etabf = \mat{\xbf\\ \ubf\\ \ybf}.
\end{align*}
And the transformed stability matrix is
\begin{align*}
\Sbf\Tbf =&
\mat{
\Phibf_{\xx} & \Sbf_{\xu} & \Phibf_{\xy} \\
\Phibf_{\ux} & \Sbf_{\uu} & \Phibf_{\uy} \\
\Sbf_{\yx} & \Sbf_{\yu} & \Sbf_{\yy}
}\Tbf\\
=&
\mat{
\Phibf_{\xx} - \Phibf_{\xy}C(zI-A)^{-1} & \Sbf_{\xu} & \Phibf_{\xy} \\
\Phibf_{\ux} - \Phibf_{\uy}C(zI-A)^{-1} & \Sbf_{\uu} & \Phibf_{\uy} \\
\Sbf_{\yx} - \Sbf_{\yy}C(zI-A)^{-1} & \Sbf_{\yu} & \Sbf_{\yy}
}.
\end{align*}

Comparing the corresponding elements and we have
\begin{align*}
\mat{\Ybf & \Wbf \\ \Ubf & \Zbf}
=&
\mat{\Sbf_{\yy} & \Sbf_{\yu}\\
\Phibf_{\uy} & \Sbf_{\uu}
}\\
=&
\mat{
C\Phibf_{\xy} + D\Phibf_{\uy} + I & \Sbf_{\yu}\\
\Phibf_{\uy} & \Phibf_{\ux}B + \Phibf_{\uy}D + I
}
\end{align*}
where
\begin{align*}
\Sbf_{\yu} = (C\Phibf_{\xx} + D\Phibf_{\ux})B + (C\Phibf_{\xy} + D\Phibf_{\uy} + I)D.
\end{align*}
Our result extends the $D = O$ case in \cite{zheng2020equivalence} to general $D$.

\paragraph{SLP and mixed parameterizations:} SLP and mixed parameterizations share the same realization \fig{realization-output-feedback}. Therefore, they also share the same stability matrix according to \lem{R-S}, i.e., $\Phibf_{\xu},\Phibf_{\uu},\Phibf_{\yx},$ and $\Phibf_{\yy}$ can be found in \eqn{of-SLS-S-1} and \eqn{of-SLS-S-2}.

\section{Corollaries: Realizations}\label{sec:realizations}
The same parameterization could admit multiple different realizations\footnote{We remark that once $\Sbf$ is fixed, $\Rbf$ is uniquely defined by \lem{R-S} (if existing). However, one parameterization may not include the whole $\Sbf$, and hence there are still some degrees of freedom for different realizations $\Rbf$.}. In this section, we consider the original state-feedback SLS realization and two alternative realization proposals for SLS. We show that the realizations can be derived from \lem{R-S} through transformations.

\subsection{Original State-Feedback SLS Realization}
SLP parameterizes all internally stabilizing controller $\Kbf$ for the state-feedback system in~\fig{realization-state-feedback}. Using the resulting $\SFpair$, SLS proposes to implement the controller as in~\fig{realization-sf-SLS}.

\begin{figure}
\centering
\includegraphics[scale=1]{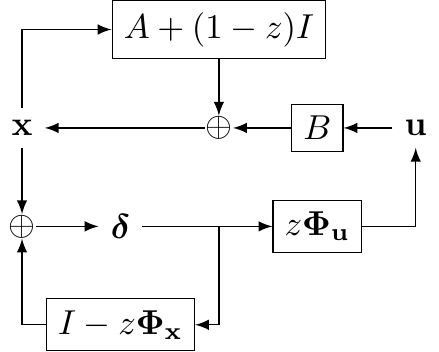}
\caption{The realization proposed in the original state-feedback SLS using the SLP $\SFpair$. By introducing an additional signal $\deltabf$, this realization avoids taking the inverse of $\Phibf_{\xbf}$.}
\label{fig:realization-sf-SLS}
\end{figure}

In other words, given $\Rbf$ as in \eqn{sf-SLS-R} and $\Sbf$ satisfying \lem{R-S}, we can realize the closed-loop system by
\begin{align*}
\Rbf_r = \mat{
A + (1-z)I & B & O\\
O & O & z\Phibf_{\ubf}\\
I & O & I - z\Phibf_{\xbf}
},
\quad
\etabf=\mat{
\xbf\\ \ubf\\ \deltabf
}.
\end{align*}

To show that, we augment \eqn{sf-SLS-R} with a dummy node
\begin{align}
I =&\ (I-\Rbf_{aug})\Sbf_{aug} \nonumber \\
=&
\mat{
zI-A & -B & O\\
-\Kbf & I & O\\
O & O & I
}
\mat{
\Phibf_{\xbf} & \Sbf_{\xu} & O \\
\Phibf_{\ubf} & \Sbf_{\uu} & O \\
O & O & I
}
\label{eqn:augmented}
\end{align}
and perform the following transformation on the augmented system to achieve the desired realization
\begin{align*}
\Tbf^{-1} =&
\mat{
I & O & O\\
\Phibf_{\ubf} & \Sbf_{\uu} & -z\Phibf_{\ubf}\\
-\Phibf_{\xbf} & -\Sbf_{\xu} & z\Phibf_{\xbf}
},\\
\Tbf =&
\mat{
I & O & O\\
O & I & \Kbf \\
z^{-1} & z^{-1}B & I - z^{-1}A
}.
\end{align*}
The realization is internally stable as
\begin{align*}
\Sbf_r =&\ \Sbf_{aug}\Tbf
=
\mat{
\Phibf_{\xbf} & \Sbf_{\xu} & \Sbf_{\xu}\Kbf \\
\Phibf_{\ubf} & \Sbf_{\uu} & \Sbf_{\uu}\Kbf \\
z^{-1} & z^{-1}B & I - z^{-1}A
}\\
=&
\mat{
\Phibf_{\xbf} & \Sbf_{\xu} & \Phibf_{\xbf}(zI-A) - I \\
\Phibf_{\ubf} & \Sbf_{\uu} & \Phibf_{\ubf}(zI-A) \\
z^{-1} & z^{-1}B & I - z^{-1}A
} \in \RHinf.
\end{align*}

\subsection{Simpler Realization for Deployment}
The original SLS realization in~\fig{realization-sf-SLS} needs to perform two convolutions $I - z\Phibf_{\xbf}$ and $z\Phibf_{\ubf}$, which are expensive to implement in practice. Therefore, \cite{tseng2020deployment} proposes a new realization in~\fig{realization-Tseng2020} that replaces one convolution by two matrix multiplications through the following corollary \cite[Theorem 1]{tseng2020deployment}.
\begin{figure}
\centering
\includegraphics[scale=1]{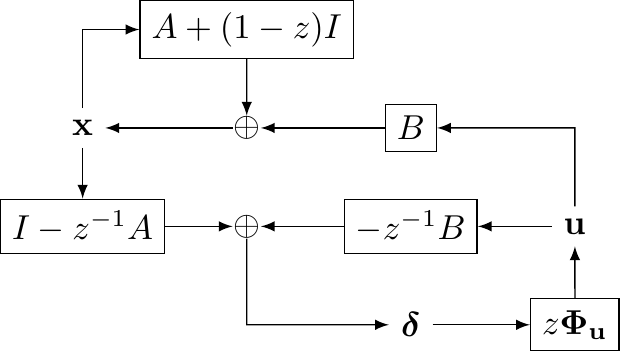}
\caption{The realization proposed in~\cite{tseng2020deployment} that realizes the SLS state-feedback controller using only one convolution $z\Phibf_{\ubf}$.}
\label{fig:realization-Tseng2020}
\end{figure}

\begin{corollary}
Let $A$ be Schur stable, the dynamic state-feedback controller $\ubf = \Kbf \xbf$ realized via
\begin{align*}
\delta\t =&\ x\t - Ax\tm - Bu\tm,\\
u\t =&\ \sum\limits_{\tau \geq 1} \Phi_{u}[\tau] \delta[t+1-\tau]
\end{align*}
is internally stabilizing.
\end{corollary}

\begin{proof}
We first write the controller realization in frequency domain:
\begin{align*}
\deltabf =&\ (I - z^{-1}A)\xbf - B\ubf,\\
\ubf =&\ z\Phibf_{\ubf}\deltabf.
\end{align*}
Together with the system, the realization is shown in~\fig{realization-Tseng2020}.

Essentially, the corollary says that given $\Rbf$ as in \eqn{sf-SLS-R} and $\Sbf$ satisfying \lem{R-S}, we can realize the closed-loop system by
\begin{align*}
\Rbf_{r} = \mat{
A+(1-z)I & B & O\\
O & O & z\Phibf_{\ubf}\\
z^{-1}(zI-A) & -z^{-1}B & O
}, \quad \etabf_{r} = \mat{\xbf\\ \ubf\\ \deltabf}.
\end{align*}

Again, we consider the augmented system in \eqn{augmented} and transform it to achieve the desired realization by
\begin{align*}
\Tbf^{-1} =&
\mat{
I & O & O\\
\Phibf_{\ubf} & \Sbf_{\uu} & -z\Phibf_{\ubf}\\
-z^{-1} & O & I
},\\
\Tbf =&
\mat{
I & O & O\\
O & \Sbf_{\uu}^{-1} & z\Sbf_{\uu}^{-1}\Phibf_{\ubf} \\
z^{-1} & O & I
}.
\end{align*}
As such, $(I-\Rbf_{r}) = \Tbf^{-1}(I-\Rbf_{aug})$ and the resulting stability matrix is
\begin{align*}
\Sbf_{r} = \Sbf_{aug}\Tbf = \mat{
\Phibf_{\xbf} & \Sbf_{\xu}\Sbf_{\uu}^{-1} & z\Sbf_{\xu}\Sbf_{\uu}^{-1}\Phibf_{\ubf} \\
\Phibf_{\ubf} & I & z\Phibf_{\ubf} \\
z^{-1} & O & I
}.
\end{align*}
Since $\Sbf_{\uu} = I + \Phibf_{\ubf}B$ is invertible, $(zI-A)\Sbf_{\xu} = B\Sbf_{\uu}$, and $A$ is Schur stable, we have
\begin{align*}
\Sbf_{\xu}\Sbf_{\uu}^{-1} = (zI-A)^{-1}B \in \RHinf,
\end{align*}
and hence the stability matrix is in $\RHinf$.
\end{proof}

In \cite{tseng2020deployment}, the authors substitute $\ubf$ into $\deltabf$ before analyzing the internal stability, which is simply another (linear) transformation of $\dbf$ and the resulting stability matrix is still internally stable.

\subsection{Closed-Loop Design Separation}
Instead of directly adopting the realization in~\fig{realization-sf-SLS}, \cite{li2020separating} found that it is possible to use much simpler transfer matrices to realize the same controller. The following corollary is from \cite[Theorem 2]{li2020separating}\footnote{To avoid the confusion with the realization matrix $\Rbf$, we write $\Pbf_c$ here instead.}.
\begin{figure}
\centering
\includegraphics[scale=1]{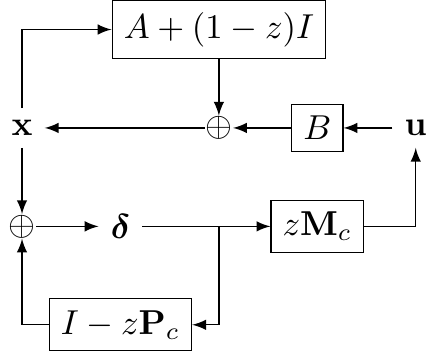}
\caption{The realization proposed in~\fig{realization-Li2020} that uses simpler transfer matrices $z\Mbf_c$ and $I-z\Pbf_c$ to implement the SLS state-feedback controller $\Phibf_{\ubf}\Phibf_{
\xbf}^{-1}$.}
\label{fig:realization-Li2020}
\end{figure}

\begin{corollary}\label{cor:Li2020}
For the causal realization in~\fig{realization-Li2020} and a given $\SFpair$ that satisfies \cor{sf-SLS}, $\{\Pbf_c,\Mbf_c\}$ realizes $\SFpair$ if and only if they satisfy
\begin{align}
\mat{
\Phibf_{\xbf} \\
\Phibf_{\ubf}
}
\mat{
zI-A & -B
}
\mat{
\Pbf_c\\
\Mbf_c
}
=
\mat{
\Pbf_c\\
\Mbf_c
}.
\label{eqn:Li2020}
\end{align}
\end{corollary}

\begin{proof}
The corollary says that for the realization
\begin{align*}
\Rbf = \mat{
A + (1-z)I & B & O\\
O & O & z\Mbf_c\\
I & O & I - z\Pbf_c
},
\quad
\etabf=\mat{
\xbf\\ \ubf\\ \deltabf
}
\end{align*}
and the given $(\Phibf_{\xbf}, \Phibf_{\ubf})$, there exists a solution
\begin{align*}
\Sbf = \mat{
\Phibf_{\xbf} & \Sbf_{\xu} & \Sbf_{\xd}\\
\Phibf_{\ubf} & \Sbf_{\uu} & \Sbf_{\ud}\\
\Sbf_{\dx} & \Sbf_{\du} & \Sbf_{\dd}
}
\end{align*}
satisfying \lem{R-S} if and only if \eqn{Li2020} holds.

If such an $\Sbf$ exists, \lem{R-S} suggests
\begin{align*}
\Sbf (I-\Rbf)
= I,
\end{align*}
and we have
\begin{align*}
(I-\Rbf)
\mat{
z\Pbf_c\\
z\Mbf_c\\
I
}
=
\mat{
zI-A & -B & O\\
O & O & O\\
O & O & O
}
\mat{
z\Pbf_c\\
z\Mbf_c\\
I
}.
\end{align*}
Therefore,
\begin{align}
&\ \Sbf(I-\Rbf)\mat{
z\Pbf_c\\
z\Mbf_c\\
I
} = \mat{
z\Pbf_c\\
z\Mbf_c\\
I
} \label{eqn:Li2020-Sdx}\\
\Rightarrow&
\mat{
\Phibf_{\xbf} \\
\Phibf_{\ubf}
}
\mat{
zI-A & -B
}
\mat{
z\Pbf_c\\
z\Mbf_c
}
=
\mat{
z\Pbf_c\\
z\Mbf_c
}\nonumber
\end{align}
and \eqn{Li2020} follows from dividing both sides by $z$.

On the other hand, when \eqn{Li2020} holds and $\SFpair$ satisfies \cor{sf-SLS}, the stability matrix can be derived from $\Sbf(I-\Rbf) = I$ as
\begin{align*}
\Sbf =
\mat{
\Phibf_{\xbf} & \Phibf_{\xbf}B & \Phibf_{\xbf}(zI-A) - I \\
\Phibf_{\ubf} & I + \Phibf_{\ubf}B & \Phibf_{\ubf}(zI-A) \\
\Sbf_{\dx} & \Sbf_{\dx}B & \Sbf_{\dx}(zI-A)
}
\end{align*}
where, by \eqn{Li2020-Sdx},
\begin{align*}
\Sbf_{\dx} = z^{-1} \Deltabf_c^{-1} = z^{-1}\left( \mat{
zI-A & -B
}
\mat{
\Pbf_c\\
\Mbf_c
} \right)^{-1}.
\end{align*}

$\Sbf$ exists if $\Sbf_{\dx}$ exists. In other words, we have to show that $\Deltabf_c$ is invertible. Since the system is causal, $I-z\Pbf_c$ and $z\Mbf_c$ are both in $\Rp$. Therefore,
\begin{align*}
\Deltabf_c = z\Pbf_c - (A\Pbf_c + B\Mbf_c) = I - z^{-1}\Jbf
\end{align*}
where $\Jbf \in \Rp$, and hence
\begin{align*}
\Deltabf_c^{-1} = I + \sum\limits_{k \geq 1}^{\infty}  z^{-k}\Jbf^k \in \Rp,
\end{align*}
which concludes the proof.
\end{proof}

We remark that \cor{Li2020} does not guarantee that $\Sbf \in \RHinf$, and hence the authors in \cite{li2020separating} propose to perform a posteriori stability check. According to the proof of \cor{Li2020}, we can easily guarantee $\Sbf \in \RHinf$ by requiring $\Sbf_{\dx} \in z^{-1}\RHinf$ (to ensure $\Sbf_{\dd} = \Sbf_{\dx}(zI-A) \in \RHinf$). This is one benefit resulting from the analysis using \lem{R-S} and condition \eqn{R-S-condition}.

\section{Corollaries: Existing Robust Results}\label{sec:existing_robust_results}
We unite the existing robust results under \lem{R-S} and \lem{perturbed-S}, including $\mu$-synthesis \cite{doyle1982analysis, doyle1985structured, zhou1998essentials}, primal-dual Youla parameterization \cite{niemann2002reliable}, input-output parameterization \cite{zheng2020sample}, and system level synthesis \cite{matni2017Scalable,anderson2019system,boczar2018finite}. In the following derivations, we use hatted symbols to denote nominal signals/matrices.

\begin{figure}
\centering
\includegraphics{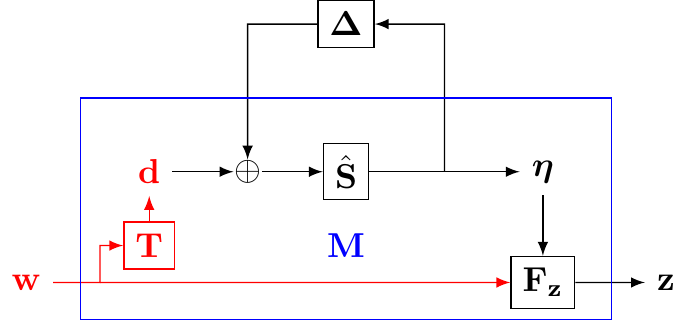}
\caption{For additive perturbation $\Deltabf$, the matrix $\Mbf$ in the $\mu$-synthesis is essentially the nominal stability matrix $\Snm$ wrapped by transformation $\Tbf$ and the output transfer matrix $\Fbf_{\zbf}$.}
\label{fig:mu-synthesis}
\end{figure}

\subsection{$\mu$-Synthesis}

$\mu$-synthesis studies the quantity ${\rm det}(I-\Mbf\Deltabf)$ \cite{doyle1982analysis, doyle1985structured, zhou1998essentials}, where the matrix $\Mbf \in \RHinf$ maps some external disturbance $\wbf$ to some output $\zbf$. The $\mu$ function, or the structured singular value, is the inverse size of the ``smallest'' structured perturbation that destabilizes $(I-\Mbf\Deltabf)^{-1}$.
Further, the original papers claim that it is always possible to choose $\Mbf$ so that $\Deltabf$ is block-diagonal. As to why ${\rm det}(I-\Mbf\Deltabf)$ is critical and why $\Deltabf$ could be formed block-diagonal, \lem{perturbed-S} provides intuitive answers.

Suppose the realization is subject to additive perturbation $\Deltabf$. We first show that $\Mbf$ is a wrapped nominal stability matrix $\Snm$.
Since $\etabf$ contains all the internal signals, the output $\zbf$ can be expressed as a linear map $\Fbf_\zbf$ of $\etabf$ and the external input $\wbf$. Meanwhile, $\wbf$ can be mapped to the external disturbance $\dbf$ through a transformation $\Tbf$. Therefore, the matrix $\Mbf$ can be expressed as in \fig{mu-synthesis}. Here the $\Stru$ is expanded by \lem{perturbed-S} as in \fig{perturbed-stability}.

Since $\Deltabf$ is an additive perturbation on the realization, we can create dummy signals to the realization so that each signal has at most one perturbed input/output. Equivalently, $\Deltabf$ is structured such that there is at most one perturbation block at each row/column. Hence, $\Deltabf$ can be permuted to take a block-diagonal form.

On the other hand, when $\wbf = \dbf$ and $\zbf = \etabf$, we have $\Mbf = \Snm \in \RHinf$. By \lem{perturbed-S}, the perturbed stability is
\begin{align*}
\Stru = \Snm(I-\Snm\Deltabf)^{-1} = \Mbf(I-\Mbf\Deltabf)^{-1},
\end{align*}
which is not stable if $I-\Mbf\Deltabf$ is not invertible, or ${\rm det}(I-\Mbf\Deltabf) = 0$.

\subsection{Robust Primal-Dual Youla Parameterization}

\cite[Chapter 3, Theorem 4.2]{tay1998high} generalizes the primal Youla parameterization $\Qbf$ in \sec{controller_synthesis-Youla} to also parameterize the stabilizable plants $\Gbf$ using dual Youla parameter $\Pbf$.\footnote{To avoid the confusion with the stability matrix $\Sbf$, we express the dual Youla parameter by $\Pbf$ instead.}
Here we present a modern rewrite of \cite[Chapter 3, Theorem 4.2]{tay1998high}:

\begin{corollary}\label{cor:primal-dual-Youla}
For the realization in \fig{realization-G-K} and the doubly coprime factorization \eqn{coprime-factorization}, let $\Gbf$ and $\Kbf$ be parameterized by
\begin{align*}
\Gbf =&\ (\Nbf_r - \Ubf_r \Pbf)(\Mbf_r-\Vbf_r \Pbf)^{-1}\\
\Kbf =&\ (\Vbf_r - \Mbf_r \Qbf)(\Ubf_r -\Nbf_r\Qbf)^{-1}
\end{align*}
where $\Pbf, \Qbf \in \RHinf$ are the dual and primal Youla parameters. The system is internally stable if and only if the following $(\Pbf,\Qbf)$ realization is internally stable:
\begin{center}
\includegraphics[scale=1]{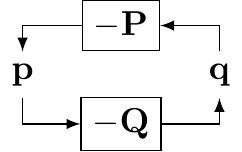}
\end{center}
\end{corollary}

\begin{proof}
Similar to the approach in \cite{tseng2021realization}, we transform the coprime factorization \eqn{coprime-factorization} by
\begin{align*}
\Tbf^{-1} =&
\mat{
(\Mbf_l-\Pbf\Vbf_l)^{-1} & O\\
O & (\Ubf_l -\Qbf\Nbf_l)^{-1} \\
}
\mat{
I & \Pbf\\
\Qbf & I
}\\
\Tbf =&
\mat{
I & \Pbf\\
\Qbf & I
}^{-1}
\mat{
\Mbf_l-\Pbf\Vbf_l & O\\
O & \Ubf_l -\Qbf\Nbf_l \\
}
\end{align*}
where
\begin{align*}
\mat{
I & \Pbf\\
\Qbf & I
}^{-1}
= \mat{
I & -\Pbf\\
-\Qbf & I
}
\mat{
(I - \Pbf\Qbf)^{-1} & O\\
O & (I - \Qbf\Pbf)^{-1}
}.
\end{align*}
The transformation leads to
\begin{align*}
& \Tbf^{-1}
\mat{
\Mbf_l & -\Nbf_l\\
-\Vbf_l & \Ubf_l
}
\mat{
\Ubf_r & \Nbf_r\\
\Vbf_r & \Mbf_r
}\Tbf\\
=& \mat{I & -\Gbf\\
-\Kbf & I}
\mat{
\Ubf_r & \Nbf_r\\
\Vbf_r & \Mbf_r
}\Tbf = I.
\end{align*}
By \lem{R-S}, the stability matrix of the system is
\begin{align*}
\mat{
\Ubf_r & \Nbf_r\\
\Vbf_r & \Mbf_r
}\Tbf
\end{align*}
which should be in $\RHinf$ to ensure internal stability. By assumption, the coprime factorization matrix, $\Pbf$, and $\Qbf$ are all in $\RHinf$. We then only need to ensure $\Tbf \in \RHinf$, which requires
\begin{align*}
\mat{
I & \Pbf\\
\Qbf & I
}^{-1} \in \RHinf.
\end{align*}
Treating the matrix inverse as a stability matrix, its corresponding realization is exactly the $(\Pbf,\Qbf)$ realization. Consequently, the whole system is internally stable if and only if the $(\Pbf,\Qbf)$ realization is internally stable, which concludes the proof.
\end{proof}

Leveraging the dual Youla parameterization, \cite{niemann2002reliable} proposes to embed the perturbation $\Deltabf$ on the plant $\Gbf$ into $\Pbf$, as such
\begin{align}
\Gbf(\Deltabf) =&\ (\Nbf_r - \Ubf_r \Pbf(\Deltabf))(\Mbf_r-\Vbf_r \Pbf(\Deltabf))^{-1}.
\label{eqn:Youla-plant}
\end{align}
\cite{niemann2002reliable} then suggests following result.

\begin{corollary}
Consider a perturbed Youla parameterization as in \cor{primal-dual-Youla} with the plant \eqn{Youla-plant}. The system is internally stable if
\begin{align}
(I - \Qbf\Pbf(\Deltabf))^{-1} \in \RHinf.
\end{align}
\end{corollary}

\begin{proof}
As shown in the proof of \cor{primal-dual-Youla}, the system is internally stable if and only if
\begin{align*}
&\mat{
I & \Pbf(\Deltabf)\\
\Qbf & I
}^{-1}\\
=& \mat{
I & -\Pbf(\Deltabf)\\
-\Qbf & I
}
\mat{
(I - \Pbf(\Deltabf)\Qbf)^{-1} & O\\
O & (I - \Qbf\Pbf(\Deltabf))^{-1}
}
\end{align*}
is in $\RHinf$. Since $\Pbf(\Deltabf), \Qbf \in \RHinf$ and
\begin{align*}
(I - \Pbf(\Deltabf)\Qbf)^{-1} = I + \Pbf (I - \Qbf\Pbf(\Deltabf))^{-1}\Qbf,
\end{align*}
we know $(I - \Qbf\Pbf(\Deltabf))^{-1} \in \RHinf$ implies $(I - \Pbf(\Deltabf)\Qbf)^{-1} \in \RHinf$, which concludes the proof.
\end{proof}

\subsection{Robust Input-Output Parameterization (IOP)}
\def\DeltaG{\Deltabf_{\Gbf}}

\cite[Theorem C.2]{zheng2020sample} presents the following robust result for IOP in \sec{controller_synthesis-IOP}.

\begin{corollary}\label{cor:robust-IOP}
For the realization in \fig{realization-G-K}, let $\Kcal_\epsilon$ be the set of robustly stabilizing controllers defined by
\begin{align*}
\Kcal_\epsilon = \left\lbrace
\Kbf : \Kbf \text{ internally stabilizes } \Gbf(\DeltaG), \forall \DeltaG \in \Dcal_\epsilon
\right\rbrace
\end{align*}
where $\Gbf(\Deltabf_{\Gbf}) = \Gbf + \DeltaG$ and $\Dcal_\epsilon = \left\lbrace \DeltaG : \norm{\DeltaG}_{\infty} < \epsilon \right\rbrace$. Then $\Kcal_\epsilon$ is parameterized by $\{\hat{\Ybf}, \hat{\Wbf}, \hat{\Ubf}, \hat{\Zbf} \}$ that satisfies \cor{IOP} and
\begin{align*}
\norm{\hat{\Ubf}}_{\infty} \leq \epsilon^{-1},
\end{align*}
and the controller is given by $\hat{\Kbf} = \hat{\Ubf} \hat{\Ybf}^{-1}$
\end{corollary}

\begin{proof}
Let
\begin{align*}
\Snm = \mat{
\hat{\Ybf} & \hat{\Wbf}\\
\hat{\Ubf} & \hat{\Zbf}
}, \quad\quad
\Deltabf = \mat{
O & \DeltaG \\
O & O
}.
\end{align*}
\lem{perturbed-S} implies that the system is internally stable if and only if
\begin{align*}
\Stru = \Snm(I - \Deltabf \Snm)^{-1} \in \RHinf
\end{align*}
for all $\Deltabf \in \Dcal_\epsilon$. Since $O \in \Dcal_\epsilon$, we need $\Snm \in \RHinf$ (which leads to the first three conditions, see \cite{tseng2021realization}) and
\begin{align*}
&(I - \Deltabf \Snm)^{-1}
=
\mat{
I - \DeltaG \hat{\Ubf} & - \DeltaG \hat{\Zbf}\\
O & I
}^{-1}\\
=&
\mat{
(I - \DeltaG \hat{\Ubf})^{-1} & (I - \DeltaG \hat{\Ubf})^{-1} \DeltaG \hat{\Zbf} \\
O & I
} \in \RHinf.
\end{align*}
Since $\DeltaG$ has a bounded $H_{\infty}$ norm, we know $\DeltaG \in \RHinf$. By the small gain theorem, $(I - \DeltaG \hat{\Ubf})^{-1} \in \RHinf$ for all $\DeltaG \in \Dcal_\epsilon$ if and only if $\norm{\hat{\Ubf}}_{\infty} \leq \epsilon^{-1}$, which concludes the proof.
\end{proof}

\subsection{Robust System Level Synthesis (SLS)}

\renewcommand{\paragraph}[1]{\vspace*{0.5\baselineskip}\noindent\textbf{#1}}
\def\DeltaSLS{\Deltabf^{\rm SLS}}

In addition to the classic SLS/SLP as in \sec{controller_synthesis-SLS}, Recent studies investigate robust controller synthesis problem with perturbed SLP under both state-feedback and output-feedback settings. We show below how to derive those results using \lem{perturbed-S}.

\paragraph{State-Feedback:} The following robust result from \cite[Theorem 2]{matni2017Scalable} and \cite[Theorem 4.3]{anderson2019system} examines the perturbed state-feedback SLP.

\begin{corollary}\label{cor:robust-sf-SLS}
For the system with realization as in \fig{realization-state-feedback}, let $(\hat{\Phibf}_\xbf, \hat{\Phibf}_\ubf, \DeltaSLS)$ be a solution to
\begin{align*}
\mat{zI - A & -B}
\mat{\hat{\Phibf}_\xbf\\ \hat{\Phibf}_\ubf} = I + \DeltaSLS,\quad  \hat{\Phibf}_\xbf, \hat{\Phibf}_\ubf \in z^{-1}\RHinf.
\end{align*}
Then, the controller realization
\begin{align*}
\hat{\deltabf}_\xbf = \xbf - \hat{\xbf}, \quad
\ubf = z \hat{\Phibf}_\xbf\hat{\deltabf}_\xbf, \quad
\hat{\xbf} = (z \hat{\Phibf}_\xbf - I)\hat{\deltabf}_\xbf
\end{align*}
internally stabilizes the system if and only if $(I + \DeltaSLS)^{-1}$ is stable. Furthermore, the actual system responses achieved are given by
\begin{align*}
\mat{\xbf \\ \ubf} =
\mat{\hat{\Phibf}_\xbf\\ \hat{\Phibf}_\ubf} (I + \DeltaSLS)^{-1} \dbf_\xbf.
\end{align*}
\end{corollary}

\begin{proof}
By definition, $(\hat{\Phibf}_\xbf, \hat{\Phibf}_\ubf, \DeltaSLS)$ is also a solution to
\begin{align*}
\mat{zI-A & -B\\
-\hat{\Kbf} & I}
\mat{\hat{\Phibf}_\xbf & \Sbf_{\xu} \\
\hat{\Phibf}_\ubf & \Sbf_{\uu}
} =
\mat{I + \DeltaSLS & O \\
O & I
}.
\end{align*}
Therefore,
by \lem{R-S}, the stability matrix for the realization in \fig{realization-state-feedback} is
\begin{align*}
&\mat{\hat{\Phibf}_\xbf & \Sbf_{\xu} \\
\hat{\Phibf}_\ubf & \Sbf_{\uu}
}
\mat{I + \DeltaSLS & O \\
O & I
}^{-1}\\
=& \mat{\hat{\Phibf}_\xbf(I + \DeltaSLS)^{-1} & \Sbf_{\xu} \\
\hat{\Phibf}_\ubf(I + \DeltaSLS)^{-1} & \Sbf_{\uu}
},
\end{align*}
which derives the system response.
As shown in \cite{tseng2021realization}, the nominal controller $\hat{\Kbf}$ is internally stabilizing if and only if
\begin{align*}
\hat{\Phibf}_\xbf(I + \DeltaSLS)^{-1}, \hat{\Phibf}_\ubf(I + \DeltaSLS)^{-1} \in z^{-1}\RHinf.
\end{align*}
In other words, since $\hat{\Phibf}_\xbf, \hat{\Phibf}_\ubf \in z^{-1}\RHinf$, $\Kbf$ is internally stabilizing if and only if $(I + \DeltaSLS)^{-1} \in \RHinf$. The internal stability of the proposed controller can then be showed by transforming the realization in \fig{realization-state-feedback} to the one in \fig{realization-sf-SLS} as done in \cite{tseng2021realization}.
\end{proof}

\paragraph{Output-Feedback:} For output-feedback systems as in \fig{realization-output-feedback}, the two corollaries below from \cite[Lemma 4.3, Corollary 4.4]{boczar2018finite} concern the controller resulting from perturbed affine space in \cor{of-SLS-zero-D}.

\begin{corollary}\label{cor:robust-of-SLS-condition}
For the system $\Gbf$ with realization as in \fig{realization-output-feedback} with $D=O$, let $\hatOFqple$ satisfy \eqn{of-SLS-constraints-1}, \eqn{of-SLS-constraints-3}, and
\begin{align}
\mat{
\hat{\Phibf}_{\xx} & \hat{\Phibf}_{\xy}\\
\hat{\Phibf}_{\ux} & \hat{\Phibf}_{\uy}
}
\mat{
zI - A \\ -C
}
=&
\mat{I + \DeltaSLS_1 \\ \DeltaSLS_2}
\label{eqn:robust-of-SLS-constraint}
\end{align}
and let the system response be given by
\begin{align*}
&\mat{
\Phibf_{\xx}(\DeltaSLS) & \Phibf_{\xy}(\DeltaSLS)\\
\Phibf_{\ux}(\DeltaSLS) & \Phibf_{\uy}(\DeltaSLS)
}\\
=&
\mat{
(I + \DeltaSLS_1)^{-1} & O\\
-\DeltaSLS_2 (I + \DeltaSLS_1)^{-1} & I
}
\mat{
\hat{\Phibf}_{\xx} & \hat{\Phibf}_{\xy} \\
\hat{\Phibf}_{\ux} & \hat{\Phibf}_{\uy}
}
\end{align*}
where by assumption $(I + \DeltaSLS_1)^{-1}$ exists and is in $\RHinf$. Then $\{\Phibf_{\xx}(\DeltaSLS)$, $\Phibf_{\ux}(\DeltaSLS)$, $\Phibf_{\xy}(\DeltaSLS)$, $\Phibf_{\uy}(\DeltaSLS)\}$ satisfies \cor{of-SLS-zero-D} for $\Gbf$. Furthermore, suppose $C$ is subject to an additive disturbance $\Deltabf_C$, i.e.,
\begin{align*}
C(\Deltabf_C) = C + \Deltabf_C,
\end{align*}
and
\begin{align*}
\DeltaSLS_1 = -\hat{\Phibf}_{\xy}\Deltabf_C,\quad\quad
\DeltaSLS_2 = -\hat{\Phibf}_{\xy}\Deltabf_C.
\end{align*}
We then have $\hatOFqple$ also satisfies \cor{of-SLS-zero-D} for the perturbed system.
\end{corollary}
\begin{corollary}\label{cor:robust-of-SLS-controller-condition}
Suppose $\hatOFqple$ satisfies the conditions in \cor{robust-of-SLS-condition}. Then, the controller
\begin{align*}
\hat{\Kbf} = \hat{\Phibf}_{\uy} - \hat{\Phibf}_{\ux} \hat{\Phibf}_{\xx}^{-1} \hat{\Phibf}_{\xy}
\end{align*}
stabilizes the system and achieves the closed-loop system response if and only if $(I + \DeltaSLS_1)^{-1} \in \RHinf$.
\end{corollary}

We now prove these two corollaries below by investigating $\Stru$.

\begin{proof}
By definition, $\hatOFqple$ is also a solution to
\begin{align*}
&\mat{
\hat{\Phibf}_{\xx} & \Sbf_{\xu} & \hat{\Phibf}_{\xy}\\
\hat{\Phibf}_{\ux} & \Sbf_{\uu} & \hat{\Phibf}_{\uy}\\
\Sbf_{\yx} & \Sbf_{\yu} & \Sbf_{\yy}
}
\mat{zI-A & -B & O\\
O & I & -\hat{\Kbf} \\
-C & -D & I}\\
=&
\mat{
I + \DeltaSLS_1 & O & O \\
\DeltaSLS_2 & I & O\\
O & O & I
}
\end{align*}
where $D = O$. Therefore, using the same transformation technique in \cite{tseng2021realization}, we can derive the nominal controller $\hat{\Kbf}$ in~\cor{robust-of-SLS-controller-condition} from
\begin{align*}
\mat{
\hat{\Phibf}_{\xx} & \hat{\Phibf}_{\xy}\\
\hat{\Phibf}_{\ux} & \hat{\Phibf}_{\uy}
}
\mat{zI-A & -B \\
-C & \hat{\Kbf}^{-1}}
=
\mat{
I + \DeltaSLS_1 & O \\
\DeltaSLS_2 & I
}.
\end{align*}
Also, by \lem{R-S}, the stability matrix for the realization in \fig{realization-output-feedback} is
\begin{align*}
&\mat{
\Phibf_{\xx}(\DeltaSLS) & \Sbf_{\xu}(\DeltaSLS) & \Phibf_{\xy}(\DeltaSLS)\\
\Phibf_{\ux}(\DeltaSLS) & \Sbf_{\uu}(\DeltaSLS) & \Phibf_{\uy}(\DeltaSLS)\\
\Sbf_{\yx}(\DeltaSLS) & \Sbf_{\yu}(\DeltaSLS) & \Sbf_{\yy}(\DeltaSLS)
}\\
=&
\mat{
I + \DeltaSLS_1 & O & O \\
\DeltaSLS_2 & I & O\\
O & O & I
}^{-1}
\mat{
\hat{\Phibf}_{\xx} & \Sbf_{\xu} & \hat{\Phibf}_{\xy}\\
\hat{\Phibf}_{\ux} & \Sbf_{\uu} & \hat{\Phibf}_{\uy}\\
\Sbf_{\yx} & \Sbf_{\yu} & \Sbf_{\yy}
}\\
=&
\mat{
(I + \DeltaSLS_1)^{-1} & O & O \\
-\DeltaSLS_2(I + \DeltaSLS_1)^{-1} & I & O\\
O & O & I
}
\mat{
\hat{\Phibf}_{\xx} & \Sbf_{\xu} & \hat{\Phibf}_{\xy}\\
\hat{\Phibf}_{\ux} & \Sbf_{\uu} & \hat{\Phibf}_{\uy}\\
\Sbf_{\yx} & \Sbf_{\yu} & \Sbf_{\yy}
},
\end{align*}
which derives the system response. Since
\begin{align*}
&\mat{
zI - A & -B
}
\mat{
\hat{\Phibf}_{\xx} & \hat{\Phibf}_{\xy}\\
\hat{\Phibf}_{\ux} & \hat{\Phibf}_{\uy}
}
\mat{
zI - A \\ -C
}\\
=&
\mat{
zI - A & -B
}
\mat{I + \DeltaSLS_1 \\ \DeltaSLS_2}\\
=&
\mat{I & O}
\mat{zI - A \\ -C},
\end{align*}
we have $(zI - A)\DeltaSLS_1 = B \DeltaSLS_2$. Therefore,
\begin{align*}
& \mat{
zI - A & -B
}
\mat{
\Phibf_{\xx}(\DeltaSLS) & \Phibf_{\xy}(\DeltaSLS)\\
\Phibf_{\ux}(\DeltaSLS) & \Phibf_{\uy}(\DeltaSLS)
}\\
=&
\mat{
zI - A & -B
}
\mat{
I + \DeltaSLS_1 & O \\
\DeltaSLS_2 & I
}^{-1}
\mat{
\hat{\Phibf}_{\xx} & \hat{\Phibf}_{\xy}\\
\hat{\Phibf}_{\ux} & \hat{\Phibf}_{\uy}
}\\
=&
\mat{
zI - A & -B
}
\mat{
\hat{\Phibf}_{\xx} & \hat{\Phibf}_{\xy}\\
\hat{\Phibf}_{\ux} & \hat{\Phibf}_{\uy}
}
\end{align*}

Also, \eqn{robust-of-SLS-constraint} and $z\hat{\Phibf}_{\ux}, \hat{\Phibf}_{\uy} \in \RHinf$ imply
\begin{align*}
\hat{\Phibf}_{\ux}(zI-A)-\hat{\Phibf}_{\uy}C = \DeltaSLS_2,
\end{align*}
and hence $\DeltaSLS_2 \in \RHinf$. Along with $(I + \DeltaSLS_1)^{-1} \in \RHinf$, we know that $\{\Phibf_{\xx}(\DeltaSLS)$, $\Phibf_{\ux}(\DeltaSLS)$, $\Phibf_{\xy}(\DeltaSLS)$, $\Phibf_{\uy}(\DeltaSLS)\}$ satisfies \cor{of-SLS-zero-D} through the same proof in \cite{tseng2021realization}.

When $C$ is perturbed, we have $\Rtru = \Rnm + \Deltabf$ where
\begin{align*}
\Deltabf = \mat{
O & O & O\\
O & O & O\\
\Deltabf_C & O & O\\
}
\end{align*}
and the result follows from \lem{perturbed-S}.

Since we have shown $\Stru \in \RHinf$ if and only if $(I + \DeltaSLS_1)^{-1}$ above, the nominal controller $\hat{\Kbf}$ stabilizes the system under the same condition, and \cor{robust-of-SLS-controller-condition} follows.
\end{proof}

\section{Applications and Beyond}\label{sec:applications}
In addition to unifying existing results, \lem{R-S} and \lem{perturbed-S} also allow us to easily derive robust results by extending the nominal results. Below, we provide some new results to demonstrate the effectiveness of the lemmas. On the other hand, although the conditions derived from \lem{R-S} and \lem{perturbed-S} in this work can unify a large portion of robust controller synthesis results in the literature, there are still results beyond its scope, which we briefly discuss at the end of this section.

\subsection{New Robust Results}
As examples, we derive the following robust results for output-feedback SLS that generalize \cor{robust-of-SLS-condition} and a condition for robust IOP.

\begin{corollary}\label{cor:robust-of-SLS-general}
For output-feedback systems as in \fig{realization-output-feedback} with arbitrary $D$, let $\hatOFqple$ satisfy \cor{of-SLS-zero-D}. Consider an additive perturbation $\Deltabf$ on the plant such that
\begin{align*}
A(\Deltabf) = A + \Deltabf_A,\quad&\quad B(\Deltabf) = B + \Deltabf_B,\\
C(\Deltabf) = C + \Deltabf_C,\quad&\quad D(\Deltabf) = D + \Deltabf_D.
\end{align*}
where $\Deltabf_A, \Deltabf_B, \Deltabf_C, \Deltabf_D \in \RHinf$. Then the nominal controller
\begin{align*}
\hat{\Kbf} = \hat{\Kbf}_0 \left( I + D \hat{\Kbf}_0\right)^{-1},
\end{align*}
where $\hat{\Kbf}_0 = \hat{\Phibf}_{\uy} - \hat{\Phibf}_{\ux} \hat{\Phibf}_{\xx}^{-1} \hat{\Phibf}_{\xy}$,
internally stabilizes the perturbed system if and only if $\Psibf \in \RHinf$ where
\begin{align*}
\Psibf = \left( I - \mat{
\Deltabf_A & \Deltabf_B\\
\Deltabf_C & \Deltabf_D
}
\mat{
\hat{\Phibf}_{\xx} & \hat{\Phibf}_{\xy}\\
\hat{\Phibf}_{\ux} & \hat{\Phibf}_{\uy}
} \right)^{-1}.
\end{align*}
\end{corollary}

\begin{proof}
Since $\hatOFqple$ satisfies \cor{of-SLS-zero-D}, we have the nominal stability $\Snm \in \RHinf$ and the nominal controller $\hat{\Kbf}$ is given by \cite[Corollary 5]{tseng2021realization}.

Given the realization in \fig{realization-output-feedback}, the perturbation $\Deltabf$ can be expressed as
\begin{align*}
\Deltabf = \mat{
\Deltabf_A & \Deltabf_B & O\\
O & O & O\\
\Deltabf_C & \Deltabf_D & O
}.
\end{align*}
Since $\Snm \in \RHinf$, \lem{perturbed-S} requires $(I-\Deltabf\Snm)^{-1} \in \RHinf$, or the inverse of
\begin{align}
I - \mat{
\Deltabf_A & \Deltabf_B & O\\
O & O & O\\
\Deltabf_C & \Deltabf_D & O
}
\mat{
\hat{\Phibf}_{\xx} & \Snm_{\xu} & \hat{\Phibf}_{\xy}\\
\hat{\Phibf}_{\ux} & \Snm_{\uu} & \hat{\Phibf}_{\uy}\\
\Snm_{\yx} & \Snm_{\yu} & \Snm_{\yy}
}
\label{eqn:robust-of-SLS-condition}
\end{align}
should be in $\RHinf$. Computing the inverse of \eqn{robust-of-SLS-condition} is equivalent to computing
\begin{align*}
\mat{
\Psibf & -\Psibf
\mat{
\Deltabf_A & \Deltabf_B\\
\Deltabf_C & \Deltabf_D
}
\mat{\Snm_{\xu}\\ \Snm_{\uu}}\\
O & I
}.
\end{align*}
We know $\Deltabf_A, \Deltabf_B, \Deltabf_C, \Deltabf_D, \Snm_{\xu}, \Snm_{\uu} \in \RHinf$.
Therefore, $\hat{\Kbf}$ still internally stabilizes the perturbed plant if and only if $\Psibf \in \RHinf$, which concludes the proof.
\end{proof}

We can extend \cor{robust-of-SLS-general} to provide an SLS version of \cor{robust-IOP}.

\begin{corollary}
For output-feedback systems $\Gbf(\Deltabf)$ as in \fig{realization-output-feedback} perturbed by structured $\Deltabf$ as in \cor{robust-of-SLS-general}, let $\Kcal_\epsilon$ be the set of robustly stabilizing controllers defined by
\begin{align*}
\Kcal_\epsilon = \{ \Kbf : \Kbf \text{ internally stabilizes } \Gbf(\Deltabf), \forall \Deltabf \in \Dcal_\epsilon \}
\end{align*}
where
\begin{align*}
\Dcal_\epsilon = \{
\Deltabf \text{ structured as in \cor{robust-of-SLS-general}} :
\norm{\Deltabf}_{\infty} < \epsilon
\}.
\end{align*}
Then $\Kcal_\epsilon$ is parameterized by $\hatOFqple$ that satisfies \cor{of-SLS-zero-D} and
\begin{align*}
\norm{
\mat{
\hat{\Phibf}_{\xx} & \hat{\Phibf}_{\xy}\\
\hat{\Phibf}_{\ux} & \hat{\Phibf}_{\uy}
}
}_{\infty} \leq \epsilon^{-1}.
\end{align*}
\end{corollary}

\begin{proof}
Similar to the proof of \cor{robust-IOP}, we know $\Deltabf \in \RHinf$ as its norm is bounded. By the small gain theorem and \cor{robust-of-SLS-general}, $\Psibf \in \RHinf$ for all $\Deltabf \in \Dcal_\epsilon$ if and only if
\begin{align*}
\norm{
\mat{
\hat{\Phibf}_{\xx} & \hat{\Phibf}_{\xy}\\
\hat{\Phibf}_{\ux} & \hat{\Phibf}_{\uy}
}
}_{\infty} \leq \epsilon^{-1},
\end{align*}
which concludes the proof.
\end{proof}

Similar to \cor{robust-sf-SLS}, \ref{cor:robust-of-SLS-controller-condition}, and \ref{cor:robust-of-SLS-general}, we can derive the following condition for the nominal IOP controller that still stabilizes a perturbed plant using \lem{perturbed-S}. The proof is mostly the same as the proof of \cor{robust-IOP} and omitted.

\begin{corollary}
For the realization in \fig{realization-G-K}, let $\{\hat{\Ybf}, \hat{\Wbf}, \hat{\Ubf}, \hat{\Zbf} \}$ satisfy \cor{IOP}. Consider an additive perturbation $\DeltaG \in \RHinf$ on the plant such that $\Gbf(\Deltabf_{\Gbf}) = \Gbf + \DeltaG$. Then the nominal controller
\begin{align*}
\hat{\Kbf} = \hat{\Ubf} \hat{\Ybf}^{-1}
\end{align*}
internally stabilizes the perturbed system if and only if
\begin{align*}
(I - \DeltaG \hat{\Ubf})^{-1} \in \RHinf.
\end{align*}
\end{corollary}

\subsection{Discussion}\label{sec:discussion}
There are still some robust results beyond the scope of the conditions from ditions from \lem{R-S} and \lem{perturbed-S}. In particular, the conditions rely on some other procedure to ensure the perturbed stability matrix is still in $\RHinf$ for the whole uncertainty set $\Dcal$, or condition \eqn{robust-S-general}.
Different $\Dcal$ may invoke different theorems. For example, while the small gain theorem works for ball-like uncertainty set, there is a line of research on bounded perturbation of transfer function coefficients, which builds upon the Kharitonov's Theorem \cite{kharitonov1978asymptotic}.
Kharitonov's Theorem suggests that robust stability over the whole uncertainty set $\Dcal$ can be achieved by the stability of $4$ elements within $\Dcal$. It is leveraged by \cite{bernstein1990robust,bernstein1992robust,shafai1992robust} to synthesize robust controllers and further generalized by \cite{chapellat1989generalization,barmish1988generalization} for larger classes of coefficient-perturbed transfer functions.

In contrast to the realization-centric perspective adopted in this work, there are also alternative approaches for system analysis using the integral quadratic constraints \cite{megretski1997system} or interval analysis \cite{patre2010robust}. Those methods also derive robust results, while we argue that the realization-centric perspective is more straightforward and the proofs are much simpler.

There is still a plenty of robust results that could be derived or generalized from \lem{R-S} and \lem{perturbed-S} but not included here, such as \cite[Proposition 1]{chen2019robust}, \cite[Lemma 2]{tsiamis2020sample}, \cite[Theorem III.5]{matni2020robust}. We encourage the reader to explore those diverse results and unify them under this paper's approach.

\section{Conclusion}\label{sec:conclusion}

In this paper, we derived the realization-stability lemma and the corresponding robust stability conditions. The realization-stability lemma is built upon our realization-centric abstraction. Unlike traditional approaches that differentiate the plant from the controller, realization abstraction examines the system as a whole, which leads to the concept of equivalent systems and provides a unified approach to synthesize and analyze controllers.

Via the realization-stability lemma, not only did we formulate the general controller synthesis problem and its robust version, but we also showed that all existing controller synthesis, realization, and some robust results are all special cases of the lemma merely concerning different realizations/system structures. In addition, we leveraged the lemmas to derive new robust results and discussed some other robust results beyond our current analysis approach.

Through these case studies, we demonstrate a unified procedure/analysis approach based on the realization-stability lemma to perform controller synthesis and realization derivation. Meanwhile, unintentionally but perhaps usefully, the paper serves as a comprehensive survey of contemporary (robust) controller synthesis results.

\bibliographystyle{IEEEtran}
\bibliography{Test}

\begin{thebibliography}{10}
\providecommand{\url}[1]{#1}
\csname url@samestyle\endcsname
\providecommand{\newblock}{\relax}
\providecommand{\bibinfo}[2]{#2}
\providecommand{\BIBentrySTDinterwordspacing}{\spaceskip=0pt\relax}
\providecommand{\BIBentryALTinterwordstretchfactor}{4}
\providecommand{\BIBentryALTinterwordspacing}{\spaceskip=\fontdimen2\font plus
\BIBentryALTinterwordstretchfactor\fontdimen3\font minus
  \fontdimen4\font\relax}
\providecommand{\BIBforeignlanguage}[2]{{%
\expandafter\ifx\csname l@#1\endcsname\relax
\typeout{** WARNING: IEEEtran.bst: No hyphenation pattern has been}%
\typeout{** loaded for the language `#1'. Using the pattern for}%
\typeout{** the default language instead.}%
\else
\language=\csname l@#1\endcsname
\fi
#2}}
\providecommand{\BIBdecl}{\relax}
\BIBdecl

\bibitem{youla1976modern1}
D.~C. Youla, J.~J. Bongiorno~Jr., and H.~A. Jabr, ``Modern {Wiener}-{Hopf}
  design of optimal controllers -- part {I}: The single-input-output case,''
  \emph{{IEEE} Trans. Autom. Control}, vol.~21, no.~1, pp. 3--13, 1976.

\bibitem{youla1976modern2}
D.~C. Youla, H.~A. Jabr, and J.~J. Bongiorno~Jr., ``Modern {Wiener}-{Hopf}
  design of optimal controllers -- part {II}: The multivariable case,''
  \emph{{IEEE} Trans. Autom. Control}, vol.~21, no.~3, pp. 319--338, 1976.

\bibitem{rotkowitz2005characterization}
M.~Rotkowitz and S.~Lall, ``A characterization of convex problems in
  decentralized control,'' \emph{{IEEE} Trans. Autom. Control}, vol.~50,
  no.~12, pp. 1984--1996, 2005.

\bibitem{sabau2014youla}
{\c{S}}.~Sab{\u{a}}u and N.~C. Martins, ``{Youla}-like parametrizations subject
  to {QI} subspace constraints,'' \emph{{IEEE} Trans. Autom. Control}, vol.~59,
  no.~6, pp. 1411--1422, 2014.

\bibitem{lessard2015convexity}
L.~Lessard and S.~Lall, ``Convexity of decentralized controller synthesis,''
  \emph{{IEEE} Trans. Autom. Control}, vol.~61, no.~10, pp. 3122--3127, 2015.

\bibitem{anderson2019system}
J.~Anderson, J.~C. Doyle, S.~H. Low, and N.~Matni, ``System level synthesis,''
  \emph{Annual Reviews in Control}, vol.~59, no.~12, pp. 3238--3251, 2019.

\bibitem{wang2019system}
Y.-S. Wang, N.~Matni, and J.~C. Doyle, ``A system level approach to controller
  synthesis,'' \emph{{IEEE} Trans. Autom. Control}, vol.~34, no.~8, pp.
  982--987, 2019.

\bibitem{wang2016localized}
Y.-S. Wang and N.~Matni, ``Localized {LQG} optimal control for large-scale
  systems,'' in \emph{Proc. {IEEE} {ACC}}, 2016.

\bibitem{anderson2017structured}
J.~Anderson and N.~Matni, ``Structured state space realizations for {SLS}
  distributed controllers,'' in \emph{Proc. Allerton}, 2017, pp. 982--987.

\bibitem{furieri2019input}
L.~Furieri, Y.~Zheng, A.~Papachristodoulou, and M.~Kamgarpour, ``An
  input-output parametrization of stabilizing controllers: Amidst {Youla} and
  system level synthesis,'' \emph{IEEE Control Systems Letters}, vol.~3, no.~4,
  pp. 1014--1019, 2019.

\bibitem{zheng2020equivalence}
Y.~Zheng, L.~Furieri, A.~Papachristodoulou, N.~Li, and M.~Kamgarpour, ``On the
  equivalence of {Youla}, system-level and input-output parameterizations,''
  \emph{{IEEE} Trans. Autom. Control}, 2020.

\bibitem{zheng2019system}
Y.~Zheng, L.~Furieri, M.~Kamgarpour, and N.~Li, ``System-level, input-output
  and new parameterizations of stabilizing controllers, and their numerical
  computation,'' \emph{arXiv preprint arXiv:1909.12346}, 2019.

\bibitem{doyle1982analysis}
J.~Doyle, ``Analysis of feedback systems with structured uncertainties,'' in
  \emph{IEE Proceedings D-Control Theory and Applications}.\hskip 1em plus
  0.5em minus 0.4em\relax IET, 1982, pp. 242--250.

\bibitem{doyle1985structured}
J.~C. Doyle, ``Structured uncertainty in control system design,'' in
  \emph{Proc. {IEEE} {CDC}}.\hskip 1em plus 0.5em minus 0.4em\relax IEEE, 1985,
  pp. 260--265.

\bibitem{zhou1998essentials}
K.~Zhou and J.~C. Doyle, \emph{Essentials of Robust Control}.\hskip 1em plus
  0.5em minus 0.4em\relax Prentice hall Upper Saddle River, NJ, 1998.

\bibitem{niemann2002reliable}
H.~Niemann and J.~Stoustrup, ``Reliable control using the primary and dual
  youla parameterizations,'' in \emph{Proc. {IEEE} {CDC}}.\hskip 1em plus 0.5em
  minus 0.4em\relax IEEE, 2002, pp. 4353--4358.

\bibitem{zheng2020sample}
Y.~Zheng, L.~Furieri, M.~Kamgarpour, and N.~Li, ``Sample complexity of {LQG}
  control for output feedback systems,'' \emph{arXiv preprint
  arXiv:2011.09929}, 2020.

\bibitem{matni2017Scalable}
N.~Matni, Y.-S. Wang, and J.~Anderson, ``Scalable system level synthesis for
  virtually localizable systems,'' in \emph{Proc. {IEEE} {CDC}}.\hskip 1em plus
  0.5em minus 0.4em\relax IEEE, 2017, pp. 3473--3480.

\bibitem{boczar2018finite}
R.~Boczar, N.~Matni, and B.~Recht, ``Finite-data performance guarantees for the
  output-feedback control of an unknown system,'' in \emph{Proc. {IEEE}
  {CDC}}.\hskip 1em plus 0.5em minus 0.4em\relax IEEE, 2018, pp. 2994--2999.

\bibitem{tseng2020deployment}
S.-H. Tseng and J.~Anderson, ``Deployment architectures for cyber-physical
  control systems,'' in \emph{Proc. IEEE ACC}, Jul. 2020.

\bibitem{tsengsubsynthesis}
------, ``Synthesis to deployment: Cyber-physical control architectures.''

\bibitem{li2020separating}
J.~S. Li and D.~Ho, ``Separating controller design from closed-loop design: A
  new perspective on system-level controller synthesis,'' in \emph{Proc. {IEEE}
  {ACC}}, 2020.

\bibitem{tseng2021realization}
S.-H. Tseng, ``Realization, internal stability, and controller synthesis,'' in
  \emph{Proc. IEEE ACC}, May 2021.

\bibitem{calafiore2005uncertain}
G.~Calafiore and M.~C. Campi, ``Uncertain convex programs: Randomized solutions
  and confidence levels,'' \emph{Mathematical Programming}, vol. 102, no.~1,
  pp. 25--46, 2005.

\bibitem{erdougan2006ambiguous}
E.~Erdo{\u{g}}an and G.~Iyengar, ``Ambiguous chance constrained problems and
  robust optimization,'' \emph{Mathematical Programming}, vol. 107, no. 1-2,
  pp. 37--61, 2006.

\bibitem{tseng2016random}
S.-H. Tseng, E.~Bitar, and A.~Tang, ``Random convex approximations of ambiguous
  chance constrained programs,'' in \emph{Proc. IEEE CDC}, Dec. 2016.

\bibitem{szabo2021generalised}
Z.~Szab{\'o}, J.~Bokor, and P.~G{\'a}sp{\'a}r, ``Generalised system level
  approach,'' \emph{IFAC-PapersOnLine}, vol.~54, no.~8, pp. 45--50, 2021.

\bibitem{tay1998high}
T.-T. Tay, I.~Mareels, and J.~B. Moore, \emph{High Performance Control}.\hskip
  1em plus 0.5em minus 0.4em\relax Springer Science \& Business Media, 1998.

\bibitem{kharitonov1978asymptotic}
V.~L. Kharitonov, ``The asymptotic stability of the equilibrium state of a
  family of systems of linear differential equations,'' \emph{Differentsial'nye
  Uravneniya}, vol.~14, no.~11, pp. 2086--2088, 1978.

\bibitem{bernstein1990robust}
D.~S. Bernstein and W.~M. Haddad, ``Robust controller synthesis using
  kharitonov's theorem,'' in \emph{Proc. {IEEE} {CDC}}.\hskip 1em plus 0.5em
  minus 0.4em\relax IEEE, 1990, pp. 1222--1223.

\bibitem{bernstein1992robust}
D.~Bernstein and W.~Haddad, ``Robust controller synthesis using kharitonov's
  theorem,'' \emph{{IEEE} Trans. Autom. Control}, vol.~37, no.~1, pp. 129--132,
  1992.

\bibitem{shafai1992robust}
B.~Shafai, M.~Monaco, and M.~Milanese, ``Robust control synthesis using
  generalized {Kharitonov's} theorem,'' in \emph{Proc. {IEEE} {CDC}}.\hskip 1em
  plus 0.5em minus 0.4em\relax IEEE, 1992, pp. 659--661.

\bibitem{chapellat1989generalization}
H.~Chapellat and S.~Bhattacharyya, ``A generalization of kharitonov's theorem;
  robust stability of interval plants,'' \emph{{IEEE} Trans. Autom. Control},
  vol.~34, no.~3, pp. 306--311, 1989.

\bibitem{barmish1988generalization}
B.~R. Barmish, ``A generalization of kharitonov's four polynomial concept for
  robust stability problems with linearly dependent coefficient
  perturbations,'' in \emph{Proc. {IEEE} {ACC}}.\hskip 1em plus 0.5em minus
  0.4em\relax IEEE, 1988, pp. 1869--1875.

\bibitem{megretski1997system}
A.~Megretski and A.~Rantzer, ``System analysis via integral quadratic
  constraints,'' \emph{{IEEE} Trans. Autom. Control}, vol.~42, no.~6, pp.
  819--830, 1997.

\bibitem{patre2010robust}
B.~M. Patre and P.~J. Deore, ``Robust state feedback for interval systems: An
  interval analysis approach,'' \emph{Reliab. Comput.}, vol.~14, pp. 46--60,
  2010.

\bibitem{chen2019robust}
S.~Chen, H.~Wang, M.~Morari, V.~M. Preciado, and N.~Matni, ``Robust closed-loop
  model predictive control via system level synthesis,'' \emph{arXiv preprint
  arXiv:1911.06842}, pp. 779--786, 2019.

\bibitem{tsiamis2020sample}
A.~Tsiamis, N.~Matni, and G.~Pappas, ``Sample complexity of {Kalman} filtering
  for unknown systems,'' in \emph{Learning for Dynamics and Control}.\hskip 1em
  plus 0.5em minus 0.4em\relax PMLR, 2020, pp. 435--444.

\bibitem{matni2020robust}
N.~Matni and A.~A. Sarma, ``Robust performance guarantees for system level
  synthesis,'' in \emph{Proc. {IEEE} {ACC}}.\hskip 1em plus 0.5em minus
  0.4em\relax IEEE, 2020, pp. 779--786.

\end{thebibliography}

\begin{IEEEbiography}[
{\includegraphics[width=1in,height=1.25in,clip,keepaspectratio]{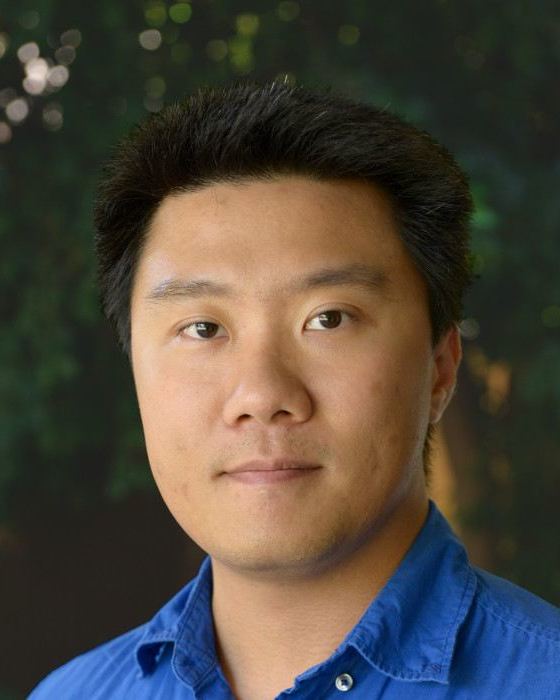}}
]{Shih-Hao Tseng}
received a Ph.D. in electrical and computer engineering from Cornell University in 2018 and worked as a postdoctoral scholar research associate at California Institute of Technology until 2021. His research interests include networked system, control theory, network optimization, and performance evaluation.
\end{IEEEbiography}

\end{document}